\documentclass[12 pt, hidelinks, leqno]{amsart}

\usepackage[utf8]{inputenc} % Required for inputting international UTF-8 characters
\usepackage[T1]{fontenc} % Output font encoding for international characters
\usepackage{babel} % Language package (language is defined in documentclass)

\usepackage{amsmath} % Create matrices, subequations and more
\usepackage{amssymb} % Math font to create symbols from characters
\usepackage{amsthm}
\usepackage{esint} %only for the mean value integral
\usepackage{mathtools} %only for the \coloneq
\usepackage[heavycircles]{stmaryrd} %for the Kulkarni-Nomizu product
\newcommand{\diff}{\mathrm{d}}
\newcommand{\R}{\mathbf{R}}

\newcommand{\N}{\mathbf{N}}
\newcommand{\Sph}{\mathbf{S}}

\newcommand{\gdo}[1]{\mathrm{O}\left( #1 \right)}

\usepackage[top=2.3cm, bottom=2cm, left=2.2cm, right=2.2cm]{geometry}
\usepackage[style=alphabetic]{biblatex}
\usepackage{csquotes} %ensure that references appear correctly when using biblatex

\usepackage[colorlinks = true,
linkcolor = red,
urlcolor  = black,
citecolor = red,
anchorcolor = black]{hyperref}
\usepackage{cleveref}
\usepackage{thmtools}

\newtheorem{theorem}{Theorem}[section]
\newtheorem{proposition}[theorem]{Proposition}
\newtheorem{lemma}[theorem]{Lemma}
\newtheorem{corollary}[theorem]{Corollary}
\newtheorem{definition}[theorem]{Definition}
\newtheorem{notation}[theorem]{Notation}
\theoremstyle{remark}
\newtheorem{remark}[theorem]{Remark}

\crefname{hyp}{hypothesis}{hypotheses}
\crefname{equation}{}{}

\title{Pohozaev identities and bubbling obstruction for Yang-Mills fields in conformal dimension}
\author{Mario Gauvrit}
\address{Mario Gauvrit
	\newline
	Universit\'e Paris Cit\'e and Sorbonne Universit\'e, CNRS, IMJ-PRG, F-75013 Paris, France.}
\email{mario.gauvrit@imj-prg.fr}
\date{}

\begin{document}
	
	\maketitle

 		\begin{abstract}
 	We study bubbling for sequences of Yang–Mills connections on closed four-manifolds and we derive a compatibility condition of Pohozaev-type between the weak limit connection and the bubble formed at a concentration point, involving the Weyl curvature of the background metric. This yields obstructions to bubbling extending earlier results of Yin \cite{yin_blow-up_2023} beyond the locally conformally flat case. As an application, we rule out certain bubbling configurations on $\mathbf{CP}^2$. 
 	\end{abstract}
	
\section*{Introduction}
The study of compactness properties of Yang–Mills connections originates in the foundational work of Uhlenbeck \cite{UhlenbeckKarenK1982CwLb}. This framework plays a central role in the analysis of the moduli space of (anti-)self-dual connections, in particular in four-dimensional geometry where they arise as absolute minimizers of the Yang–Mills functional.

A major development in this direction was achieved by Donaldson \cite{Donaldson}, who used these compactness results to construct a compactification of the moduli space of instantons with topological charge $1$. In this setting, bubbling phenomena are completely understood: sequences of instantons may degenerate by energy concentration at isolated points, leading to the formation of bubbles modelled on the standard 1-instanton on $\Sph^4$, as described via Taubes’ gluing theory \cite{taubes82}.

For higher topological charge, however, the structure of the resulting bubble-tree compactification remains only partially understood. In particular, a precise description of how energy decomposes between the weak limit connection and the emerging bubbles is still not fully understood. A natural question is whether there exist compatibility constraints between the limiting connection and the bubbles that appear in the blow-up process.

The main result of this paper shows that, in a bubbling configuration, the weak limit connection and the bubble are not independent. Instead, they satisfy a nontrivial compatibility condition at the concentration point, involving both the curvature of the limit connections and the conformal geometry of the underlying manifold and such constraints can rule out certain bubbling configurations.

This type of interaction is expected from analogy with the theory of harmonic maps from $\Sph^2$ to $\Sph^2$, which are either holomorphic or anti-holomorphic, preventing an anti-holomorphic bubble from developing on a holomorphic weak limit. In the Yang–Mills setting, related phenomena appear in the work of Yin \cite{yin_blow-up_2023} on the blow-up analysis of Yang–Mills fields in dimension four under a local conformal flatness condition: in the case of one bubble, algebraic constraints on curvature at concentration points are obtained. As a consequence no $1$-instanton can blow up on a $-1$-instanton.

The aim of this paper is to extend such obstruction results beyond the locally conformally flat case. In particular, we investigate conditions under which certain bubbling configurations are ruled out, with applications to compact four-manifolds such as $\mathbf{CP}^2$.

Let $(M^4,h)$ be a closed oriented manifold and $G$ be a compact Lie group. The Yang-Mills functional is defined for connections $A$ on principal $G$-bundles over $M$ by \[ \mathcal{YM}(A) = \int_M |F_A|^2_h \mathrm{vol}_h,\] where $F_A$ is the curvature of $A$. By definition, Yang-Mills connections are the critical points of this functional. In general, sequences of Yang-Mills connections with bounded energy are subject to concentration phenomena which we recall below in the case of one bubble for simplicity.
	
		\begin{definition} Let $(A_k)_k$ be a sequence of Yang-Mills connections on $(M^4,h)$, $A_\infty$ and $\widehat{A}_\infty$ Yang-Mills connections on $M$ and $\R^4$ respectively, $p$ a point in $M$ and $(\lambda_k)_k$ a sequence of positive numbers with $\lim \lambda_k = 0$.  We say that $(A_k)_k$ bubble-tree converges to $(A_\infty, \widehat{A}_\infty)$  with concentration at $p\in M$ of scale $(\lambda_k)_k$ if \begin{enumerate}
        \item $A_k$ converges to $A_\infty$ in $\mathcal{C}^\infty_{\mathrm{loc}}(M\setminus \{p\})$, up to gauge transformations,
        \item there exists $(p_k)_k$ with $p_k\to p$ such that $\phi_{k}^*A_k$ converges to $\widehat{A}_\infty$ in $\mathcal{C}^\infty_{\mathrm{loc}}(\R^4)$, up to gauge transformations, 
        where $\phi_{k} (x)=p_k+ \lambda _k x$ in local coordinates,
        \item 
        $$ \lim_{\delta\to 0}\lim_{k\rightarrow +\infty}\int_{\mathrm{B}_\delta(p_k)\backslash\mathrm{B}_{\lambda_k/\delta}(p_k)} \vert F_{A_k}\vert^2\, \mathrm{vol}_h  = 0$$
    \end{enumerate}
		    \label{bubbletree}
		\end{definition} The reader can refer for instance to \cite{taubes_path-connected_1984,taubes_framework_1988,laurainriviere2014angular, feehan_discreteness_2015,rivière2015variations} for a description of the general picture. In the following, for the sake of clarity, $p=p_k$, denoted by $0$ in normal coordinates. In the single bubble setting, our main result is the following (see \cref{mainthm4} for the general bubble-tree case):
		
		\begin{theorem}
		    \label{mainthm}
		     Under the conditions of \cref{bubbletree} and in the gauge of \cref{Rivieregauge}, the $2$-tensor $\mathrm{P}$ defined below is traceless symmetric:
		    \begin{equation*} \mathrm{P} = 
		    \left( F_{A_\infty} \circ \psi^*{F}_{\widehat{A}_\infty} \right)(0) + \int_{\R^4}  \langle  S_{{F}_{\widehat{A}_\infty}} \otimes \mathrm{W} , \mathrm{T} \rangle \mathrm{vol}
		    \end{equation*} Here $\psi:x \mapsto x/|x|^2$ is the inversion in $\R^4$, $S_{{F}_{\widehat{A}_\infty}}$ is the stress-energy tensor of the bubble (see \cref{defstresstensor}), $\mathrm{W}$ is the Weyl tensor of $h$ at $0$ and $\mathrm{T}$ is an explicit $(2,6)$-tensor (see \cref{mainthm4}).
		\end{theorem} 	
		
		This constraint is a Pohozaev identity, analogous to the balancing law between the radial part and the angular part of the energy of harmonic maps. It is obtain in the same manner, by integrating the PDE against a suitable solution of its linearisation. The origin of such infinitesimal balancing conditions is the conformal invariance of the problem which implies, by Noether's theorem, the existence of conservation laws involving the stress-energy tensor. \Cref{mainthm} is the result of a careful analysis of such quantities where the two parts of the bubble-tree interact, namely in the middle of the neck regions. Note that additional constraints can be obtained from gauge invariance, see \cref{mainthm5}.
		Such ideas proved to be useful to rule out bubbling in various contexts: Yamabe metrics, e.g. \cite{schoen2006variational,druet_compactness_2004}; Ginzburg-Landau \cite{gianocca2025rigidityginzburglandauapproximationharmonic}; Yang-Mills \cite{yin_blow-up_2023}. 
		 \Cref{mainthm} generalises \cite[Theorem 1.4]{yin_blow-up_2023} where $\mathrm{W}=0$. In the case where $\widehat{A}_\infty$ is either self-dual or anti-self-dual, its stress-energy tensor vanishes and therefore the obstructions obtained in \cite[Theorem 1.4]{yin_blow-up_2023} still hold in this level of generality:
		 \begin{corollary}
		 	\label{coro3}
		 Under the conditions of \cref{mainthm}, if we further assume that $\widehat{A}_\infty$ is SD or ASD then the $2$-tensor $\mathrm{P}$ defined by \begin{equation*} \mathrm{P} = 
		    \left( F_{A_\infty} \circ \psi^*{F}_{\widehat{A}_\infty} \right)(0)
		    \end{equation*} is traceless symmetric.
		 \end{corollary}
	\noindent	 A direct consequence of our analysis in the non-conformally flat setting is the following: 
		 \begin{corollary}
	Under the assumptions of \cref{mainthm}, in the case $M=\mathbf{CP}^2$ and $G=\mathrm{SU}$(2), it is not possible for $A_\infty$ and $\widehat{A}_\infty$ to be respectively a $1$-instanton and a $-1$-instanton.
	\end{corollary} 
	
	\noindent For the sake of the presentation we have in the introduction considered the case of only one bubble, but our result easily generalizes to the general bubble-tree case. For instance if we consider the improved definition of a bubble-tree given in \cref{bbobstruction}, so we have a limit solution, with some branches, see remark \ref{branch}, each branch has a bottom bubble. Then one consequence of \cref{coro3} is that the bottom bubble determines the behaviour of the whole branch, see \cref{bbobstruction} for a precise statement.
	
	\begin{theorem} In the case of a general bubble-tree, in a branch, if all bubbles are SD or ASD with absolute degree $1$, then the bottom bubble determines the sign of the whole branch, in other words, if the bottom bubble is SD (resp. ASD) then all the bubbles of the branch are SD (resp. ASD).
	    \label{thmgeneralbubble}
	\end{theorem} 
	
\noindent	\textbf{Organisation of the paper.} We first recall the definition and relevant properties of the stress-energy tensor in the Yang-Mills setting and formulate the corresponding Pohozaev identities. We then establish a pointwise asymptotic expansion of the curvature in the neck regions, relying on Rivière's gauge. In \cref{proofmainthm}, these ingredients are combined to prove the main theorem: we treat separately the constraints coming from infinitesimal conformal symmetry and from infinitesimal gauge symmetry, and extend the result to the general bubble-tree setting. \Cref{bbobstruction} is devoted to applications: rigidity statements along the branches of a bubble-tree and the exclusion of certain bubbling configurations on $\mathbf{CP}^2$. Technical results are collected in the appendix.\\

\noindent	\textbf{Added in proof.} During the final preparation of this manuscript, Waldron and Yin obtained related results independently \cite{waldronyin2026}. Their Theorem 1.1 overlaps with \cref{coro3}. In fact, under the additional assumption that the limit is unobstructed, they obtain a stronger conclusion. In contrast, \cref{mainthm} applies to general sequences of Yang–Mills connections, which is more general than the SD/ASD case. 	\\

\noindent	\textbf{Acknowledgements.} 
I would like to express my sincere gratitude to my supervisor, Paul Laurain, for his valuable advice, his support, and his careful proofreading of this paper.\\

\noindent	\textbf{Miscellaneous notations:} \begin{itemize}
	    \item $F\circ_h G (X,Y) = \langle \iota_X F, \iota_Y G \rangle_h$,
	    \item $\xi$ denotes the flat metric on $\R^4$,
	   \item when the metric is omitted, we mean the flat one.
	   \item if $T$ is a 2-tensor, $T_{[ij]} \coloneq T_{ij} - T_{ji}$
	\end{itemize}  
		
	\section{The stress-energy tensor}
	
	\subsection{Definition and first properties}

In this section, we recall general properties of the stress-energy tensor, see \cite{XinMonotonicity}.

\begin{definition}\label{defstresstensor} Let $F\in \Gamma(\Lambda^2T^*M \otimes \mathfrak{g})$. The stress-energy tensor of $F$ is the symmetric $2$-tensor denoted by $S_{F,h}$ and defined as \[ S_{F,h} = \frac{1}{4}|F|^2_h h - F{\circ}_h F.\]
\end{definition}
 \begin{remark}
    Up to a sign, $ S_{F,h}$ is the traceless part of $F{\circ}_h F$.
\end{remark}
\begin{remark} Consider $F\in \Gamma(\Lambda^2T^*M \otimes \mathfrak{g})$ and $g:M\mapsto G$. Then $S_{F^g,h} = S_{F,h}$ where $F^g\coloneq g^{-1} F g$.
\end{remark}

We now state an integration by parts formula for the stress-energy tensor which will be crucial in the following.
\begin{proposition}[(2.14) in \cite{XinMonotonicity}] \label{pohozaev1} Consider a domain $D\subset M$ with smooth boundary. Denote by $\nu$ its unit outward normal vector field. For every $X\in \Gamma(T M)$ and $F\in \Gamma(\Lambda^2T^*M \otimes \mathfrak{g})$, \begin{equation*}
		 \int_D \left( \frac{1}{2} \langle S_{F, h}, \mathcal{L}_X h \rangle_h + \mathrm{div}_h S_{F,h} (X) \right) \mathrm{vol}_h =
			\int_{\partial D} S_{F, h} (X, \nu) {\mathrm{vol}_{h}}_{| \partial D}.
		\end{equation*} 
\end{proposition}
In the context of Yang-Mills fields, this formula can be simplified using the following property:
\begin{proposition}[(2.10) in \cite{XinMonotonicity}] The stress-energy tensor associated to the curvature of a Yang-Mills connection is divergence-free.\label{YMdiv0}
\end{proposition}

\noindent Recall that the Hodge operator $\star_h: \Lambda^2 T^*M \to \Lambda^2 T^*M$ is an isometry satisfying $\star_h ^2 = \mathrm{id}$ hence induce an orthogonal splitting \begin{equation*}
    \Lambda^2 T^*M = \Lambda^+ T^*M \oplus \Lambda^- T^*M
\end{equation*} where $\Lambda^\pm T^*M$ is the $\pm 1$ eigenspace of $\star_h$. Sections of $\Lambda^+ T^*M$ (resp. $\Lambda^- T^*M$) are called self-dual forms (resp. anti-self-dual forms) and if $\omega$ is a section of $\Lambda^2 T^*M$, we write $\omega = \omega^{+,h} + \omega^{-,h}$ its decomposition according to this splitting.

	\begin{proposition}\label{stressSDASD} For every $F\in \Gamma(\Lambda^2T^*M \otimes \mathfrak{g})$, 
		\begin{equation*} S_{F, h} =-2 F^{+,h} {\circ}_h F^{-,h}. \end{equation*} In particular, if either $F^{+,h}=0$ or $F^{-,h}=0$ then $S_{F,h}=0$. 
	\end{proposition}

\begin{proof}
		From \cref{formsinteriorproduitscalaire}, $S_{F,h} = - S_{\star_h F,h}$ and since $\star_h$ is an isometry and $\circ_h$ is bilinear \begin{align*}
			S_{F,h} &= \frac12 (S_{F,h} - S_{\star_h F,h}) \\
			&=- \frac12 (F \circ_h F -  \star_h F \circ_h \star_h F)\\
			&= -\frac12 ( (F^{+,h} + F^{-,h}) \circ_h (F^{+,h}+F^{-,h}) - (F^{+,h}-F^{-,h}) \circ_h (F^{+,h}-F^{-,h}))\\
			&= - \left(F^{+,h} \circ_hF^{-,h} +F^{-,h} \circ_hF^{+,h}  \right)
		\end{align*} Finally, use  \cref{formsinteriorproduitscalaire} again to check that $F^{+,h} \circ_hF^{-,h} = F^{-,h} \circ_hF^{+,h}$.
	\end{proof}
	\begin{remark}
		It is also true that being the curvature of a Yang-Mills $\mathrm{SU}(2)$-connection with vanishing stress-energy 
		tensor is equivalent to being SD/ASD (see \cite{girardi1978self}). 
	\end{remark}

\subsection{Pohozaev identities}
\begin{notation} We assume that the bundle has been trivialised over a ball in geodesic coordinates.
\end{notation}

\begin{remark} Denote by $\xi$ the flat metric on $\R^4$. We use the identification of square matrices of size $4$ with $2$-tensors on $\R^4$: a matrix $T$ corresponds to the bilinear form $\xi(\cdot, T\cdot)$.
Recall that if $T=\xi$ or $T$ is skew-symmetric, the vector field $X_T \coloneq T r\partial_r$ is conformal that is $X_T$ generates a flow of conformal transformations for the metric $\xi$ (in this context, of scaling and rotations respectively). For this reasons, we will denote by $\mathfrak{conf}(4)$ the space $\R \xi \oplus \Lambda^2 \R^4$.
\end{remark}
\noindent  The identities of the type of \cref{pohozaev1} with $X=X_T$ can all be accounted for in the following compact form:

\begin{proposition}\label{tensorpohozaev} Consider a ball $B\subset M$ in geodesic coordinates. If $S$ is the stress-energy tensor of a Yang-Mills connection on $(M,h)$ the tensor $\mathrm{P}$ defined below is traceless symmetric: \begin{align} \label{deftensorpohozaev} \mathrm{P}&=
\int_{\partial B} \left(\iota_{\partial_r}S\otimes r\diff r\right) {\mathrm{vol}_{h}}_{| \partial B}\\
&\quad- \int_{B} \left(\frac{1}{2} \langle  h^{-1}\nabla h, h^{-1}S \rangle \otimes r \diff r + S\circ (h^{-1}-\xi) + \frac{1}{4}\langle \xi - h^{-1},  S\rangle \xi \right)\mathrm{vol}_h\notag
\end{align}
where the integrals are understood coefficient-wise in the geodesic coordinates.
\end{proposition}

\begin{proof} Apply \cref{pohozaev1} with $D=B$ and $X=X_T$ where $T\in \mathfrak{conf}(4)$ is arbitrary, using \cref{YMdiv0}. Note that $\nu = \partial_r$ on $\partial B$ according to Gauss's lemma.
 Write $T_{i j} = \xi (\partial_i, T\partial_j)$. Then \begin{align*}
    0 &=  \int_{\partial B} S (T r\partial_r, \partial_r) {\mathrm{vol}_{h}}_{| \partial B} - \frac12 \int_B \langle S, \mathcal{L}_{Tr\partial_r} h \rangle_h \mathrm{vol}_h \\
    &= T_{ij} \left( \int_{\partial B} S (x^j\partial_i, \partial_r) {\mathrm{vol}_{h}}_{| \partial B} - \frac12\int_B \langle S, \mathcal{L}_{x^j\partial_i} h \rangle_h \mathrm{vol}_h \right)
\end{align*} and this means that the $2$-tensor $\Phi$ given by \begin{equation*}
    \Phi_{ij} = \int_{\partial B} S (x^j\partial_i, \partial_r) {\mathrm{vol}_{h}}_{| \partial B} - \frac12\int_B \langle S, \mathcal{L}_{x^j\partial_i} h \rangle_h \mathrm{vol}_h 
\end{equation*} is orthogonal to $\mathfrak{conf}(4)$, that is $\Phi$ is traceless symmetric. On the one hand, it is straightforward that \begin{equation}
    S (x^j\partial_i, \partial_r) = \left(\iota_{\partial_r}S\otimes r\diff r\right)_{ij}. \label{tensorisationbdy}
\end{equation} On the other hand, writing \begin{equation}
    h=h_{\alpha\beta} \diff x^{\alpha}\otimes\diff x^{\beta},
\end{equation}  we get the following expression \begin{align*}
    \mathcal{L}_{x^j\partial_i} h &=  \left(\mathcal{L}_{x^j\partial_i}h_{\alpha\beta} \right)\diff x^{\alpha}\otimes\diff x^{\beta} + h_{\alpha\beta} \left(\mathcal{L}_{x^j\partial_i}\diff x^{\alpha}\right)\otimes\diff x^{\beta} + h_{\alpha\beta}\diff x^{\alpha} \otimes \left(\mathcal{L}_{x^j\partial_i}\diff x^{\beta}\right)\\
    &= \left(x^j\partial_i h_{\alpha\beta} \right)\diff x^{\alpha}\otimes\diff x^{\beta} + h_{\alpha\beta} \delta_{i\alpha}\diff x^j  \otimes\diff x^{\beta} + h_{\alpha\beta}\diff x^{\alpha} \otimes \delta_{i\beta}\diff x^j\\
    &= \nabla_{x^j\partial_i}h + h_{i\beta}\diff x^j  \otimes\diff x^{\beta} + h_{\alpha i}\diff x^{\alpha} \otimes \diff x^j
\end{align*}
which implies \begin{align*}
    \langle S, \mathcal{L}_{x^j\partial_i} h \rangle_h - \langle S, \nabla_{x^j\partial_i} h \rangle_h &= h_{i\beta}  \langle S, \diff x^j  \otimes\diff x^{\beta} \rangle_h+  h_{i\alpha} \langle S, \diff x^i  \otimes\diff x^{\alpha} \rangle_h \\
    &= h_{i\beta} h^{j\mu}S_{\mu \nu}h^{\nu\beta} + h_{i\alpha} h^{j\mu}S_{\mu \nu}h^{\nu\alpha}\\
    &= 2 h^{j\mu}S_{\mu i}.
\end{align*} Finally, we obtain \begin{equation} \label{tensorisationvol}
     \langle S, \mathcal{L}_{x^j\partial_i} h \rangle_h = \left(\langle h^{-1}S, h^{-1}\nabla h \rangle \otimes r\diff r + 2 S\circ h^{-1}\right)_{ij}.
\end{equation}
Therefore, if $\mathrm{P}$ is the tensor defined in \cref{deftensorpohozaev}, the combination of \cref{tensorisationbdy,tensorisationvol} yields \begin{equation}
    \mathrm{P} = \Phi + \int_B \left( S\circ \xi - \frac14 \langle \xi - h^{-1}, S\rangle \xi\right)\mathrm{vol}_h.
\end{equation} However, $ S\circ \xi= S$ and $\langle \xi - h^{-1}, S\rangle   =\mathrm{tr}(S) - \mathrm{tr}_h(S) = \mathrm{tr}(S)$. Since we have  \begin{equation}
    \mathrm{P} -  \Phi = \int_B \left( S -\frac14  \mathrm{tr}(S)\xi\right)\mathrm{vol}_h
\end{equation} then $ \mathrm{P} -  \Phi $ is traceless symmetric, and so is $\mathrm{P}$.
\end{proof}
	
	\begin{lemma}\label{simplifypohozaevtensor} Consider a sequence $(A_k)_k$ as in \cref{bubbletree} with concentration scale $\lambda_k$, in the gauge of \cref{Rivieregauge}. Define $\widehat{A}_k= \lambda_k A_k(\lambda_k\cdot)$ and denote by $\mathrm{P}_k$ the Pohozaev tensor of \cref{tensorpohozaev} for $A_k$ in the ball $B= \mathrm{B}_{\sqrt{\lambda_k}}$ in geodesic coordinates. Consider $\gamma$ the degree 2-term of the Taylor expansion of $h$ in geodesic coordinates at $0$. Then for every $k\in \N$ and  $\delta\in(0,1)$, the following holds\begin{align} 
    \lambda_k ^{-2} \mathrm{P}_k = &\int_{\Sph^3} \left(\iota_{\partial_r}S_{F_{A_k}}(\sqrt{\lambda_k}\cdot) \otimes r\diff r\right) \mathrm{vol}_{\Sph^3}\notag\\ 
    &- \int_{\mathrm{B}_{1/\delta}} \left(\frac{1}{2} \langle  S_{F_{\widehat{A}_k}} , \nabla\gamma\rangle \otimes r \diff r -S_{F_{\widehat{A}_k}} \circ \gamma + \frac{1}{4}\langle S_{F_{\widehat{A}_k}}, \gamma \rangle \xi \right)\mathrm{vol} + \varpi_{k,\delta} \label{approxlocalPoh}
\end{align}  with \begin{equation*}
    \lim_{\delta\to 0}\lim_{k\to +\infty} \varpi_{k,\delta} = 0.
\end{equation*}
\end{lemma}

\begin{proof} In this context \cref{tensorpohozaev} gives:
\begin{align}
    \lambda_k ^{-2} \mathrm{P}_{k}& =
 \lambda_k ^{-2}\int_{\partial \mathrm{B}_{\sqrt{\lambda_k}}} \left(\iota_{\partial_r}S_k\otimes r\diff r\right) {\mathrm{vol}_{h}}_{| \partial \mathrm{B}_{\sqrt{\lambda_k}}} \label{auxPoh}\\&\quad - \lambda_k ^{-2} \int_{\mathrm{B}_{\sqrt{\lambda_k}}} \left(\frac{1}{2} \langle  h^{-1}\nabla h h^{-1}, S_k \rangle \otimes r \diff r + S_k\circ (h^{-1}-\xi) + \frac{1}{4}\langle \xi - h^{-1},  S_k\rangle \xi \right)\mathrm{vol}_h \notag
\end{align} where $S_k = S_{F_{A_k},h}$. The key point is the following estimate: \begin{align}
    \lambda_k^{-2}\int_{\mathrm{B}_{\sqrt{\lambda_k}}} r^3 |F_{A_k}|^2 \mathrm{vol}_\xi &\leq C  \left(\lambda_k ^{-2} \int_{\lambda_k}^{\sqrt{\lambda_k}} r^3\left(\frac{\lambda_k ^2}{r^4}\right) ^2  r^3\diff r + \int_{\mathrm{B}_{{\lambda_k}}} \lambda_k |F_{A_k}|^2 \mathrm{vol}_\xi  \right)\notag\\
    &\leq C\lambda_k \label{estaux1}
\end{align} where we used the estimate of \cref{Rivieregauge} in the neck. This means that every metric term of order $\gdo{r^3}$ in the integrand of the last term of \cref{auxPoh} can be dropped in the limit $k\to+\infty$. More precisely, since by definition \begin{equation}
    h = \xi + \gamma + \gdo{r^3}
\end{equation} we can replace in \cref{auxPoh} $h^{-1}-\xi$ by $-\gamma$, $h^{-1}\nabla h h^{-1}$ by $\nabla \gamma$, $\mathrm{vol}_h$ by $\mathrm{vol}\coloneq \mathrm{vol}_\xi$ and $S_k$ by $S_{F_{A_k}}\coloneq S_{F_{A_k},\xi}$ while creating an error of order $\gdo{\lambda_k}$, that is \begin{align*}
    \lambda_k ^{-2} \mathrm{P}_{k}& =
 \lambda_k ^{-2}\int_{\partial \mathrm{B}_{\sqrt{\lambda_k}}} \left(\iota_{\partial_r}S_k\otimes r\diff r\right) {\mathrm{vol}_{h}}_{| \partial \mathrm{B}_{\sqrt{\lambda_k}}} \\&\quad - \lambda_k ^{-2} \int_{\mathrm{B}_{\sqrt{\lambda_k}}} \left(\frac{1}{2} \langle S_{F_{A_k}}, \nabla\gamma\rangle \otimes r \diff r - S_{F_{A_k}}\circ \gamma + \frac{1}{4}\langle \gamma,  S_{F_{A_k}}\rangle \xi \right)\mathrm{vol}_\xi + \gdo{\lambda_k}
\end{align*}

\noindent Similarly, for all $\delta\in (0,1)$, \begin{equation}
    \lambda_k^{-2}\int_{\mathrm{B}_{\sqrt{\lambda_k}}\backslash \mathrm{B}_{{\lambda_k}/\delta}} r^2 |F_{A_k}|^2 \mathrm{vol}_\xi \leq C  \int_{\lambda_k/\delta}^{\sqrt{\lambda_k}} \lambda_k ^2 r^{-3}  \diff r \leq C \delta^2 \label{estaux2}
\end{equation} and the contribution of the neck region in the last term of \cref{auxPoh} also vanishes in the limit. 

\noindent Finally, using \cref{Rivieregauge} again, $F_{A_k}$ is of order $\gdo{1}$ at scale $\sqrt{\lambda_k}$ so we can also replace $S_k$ by $S_{F_{A_k}}$ and ${\mathrm{vol}_{h}}_{| \partial \mathrm{B}_{\sqrt{\lambda_k}}}$ by ${\mathrm{vol}_{\xi}}_{| \partial \mathrm{B}_{\sqrt{\lambda_k}}}$. We then obtain \begin{align*}
    \lambda_k ^{-2} \mathrm{P}_{k}& =
 \lambda_k ^{-2}\int_{\partial \mathrm{B}_{\sqrt{\lambda_k}}} \left(\iota_{\partial_r}S_{F_{A_k}}\otimes r\diff r\right) {\mathrm{vol}_{\xi}}_{| \partial \mathrm{B}_{\sqrt{\lambda_k}}} \\&\quad - \lambda_k ^{-2} \int_{\mathrm{B}_{\sqrt{\lambda_k}}} \left(\frac{1}{2} \langle S_{F_{A_k}}, \nabla\gamma\rangle \otimes r \diff r - S_{F_{A_k}}\circ \gamma + \frac{1}{4}\langle \gamma,  S_{F_{A_k}}\rangle \xi \right)\mathrm{vol}_\xi + \gdo{\lambda_k + \delta^2}
\end{align*}
The formula \cref{approxlocalPoh} is then obtained using the change of variables $y=\sqrt{\lambda_k}x$ in the first integral and $y=\lambda_k x$ in the second one (using the homogeneity of $\gamma$), and by calling $\varpi_{k,\delta}$ the term $\gdo{\lambda_k + \delta^2}$.
\end{proof}
\begin{remark} As mentioned in the proof of \cref{simplifypohozaevtensor}, on the right-hand side of \cref{approxlocalPoh}, every term is bounded. The goal of next section is to compute their limits to obtain \cref{mainthm}.
\end{remark}
	
\section{Asymptotic expansion of the curvature in the neck region}

		To prove \cref{mainthm}, we combine the conservation laws to pointwise estimates of the curvature in the neck regions, as in \cite{yin_blow-up_2023}. For this purpose, but also to identify the fibers of the bundles at the concentration points (which is needed to make sense of the tensor $\mathrm{P}$), we must fix a gauge in the neck region. As opposed to \cite{yin_blow-up_2023}, we do not construct an \textit{ad hoc} gauge for our problem and choose to use the one constructed by Rivière in \cite{riviere2002interpolation} (see also \cite{rivière2015variations}). We recall below some of its properties.
			\begin{theorem}
			\label{Rivieregauge} 
			Under the conditions of \cref{bubbletree}, there exists a gauge such that in $\mathrm{B}_\delta(p_k)\backslash\mathrm{B}_{\lambda_k}(p_k)$,		     $\diff^{\ast} A_k =
		0$ and
		\begin{equation*} r | A_k | + r^2 | \nabla A_k | + r^3 | \nabla^2 A_k | \leq C \left(
		r^2+\frac{\lambda_k^2}{r^2} \right).\end{equation*} In particular \begin{equation*} | F_{A_k} | \leq C \left(
		1+\frac{\lambda_k^2}{r^4} \right)\end{equation*} where $r = |x-p_k|$.
		\end{theorem} \noindent As pointed out in \cite[Theorem 1.1]{yin_blow-up_2023}, the identification of the fibers at the concentration point does not depend on the choice of gauge satisfying the above mentioned properties.
		
	\begin{theorem}\label{asymptoticneckexpansion} Consider a sequence $(A_k)_k$ as in \cref{bubbletree} with bubble-tree limit $(A_\infty, \widehat{A}_\infty)$ and concentration scale $\lambda_k$, in the gauge of \cref{Rivieregauge}. The following holds	
		\begin{equation*} \sup_{\partial \mathrm{B}_{\sqrt{\lambda_k}}} | F_{A_k} - F_{A_{\infty}} - \lambda_k^{-2}
		 F_{\hat{A}_{\infty}}(\lambda_k^{-1} \cdot) | \underset{k
			\rightarrow + \infty}{\rightarrow} 0. \end{equation*}
	\end{theorem}
	
	\noindent In the following we will use the following notation: \begin{equation*}
	    \omega_{\lambda}(x) \coloneq |x| + \lambda{|x|}^{-1}.
	\end{equation*}
	
	\begin{proposition} \label{key} For every $\alpha\in (2,3)$,  there exists $C > 0$ with the following property: for all $\lambda \in (0, 1
		/ 2)$ and $e \in \mathrm{W}^{1, 2} (\mathrm{B}_1 \backslash
		\mathrm{B}_{\lambda}, \Lambda^1 \mathbf{R}^4 \otimes \mathfrak{g})$, if $e$
		satisfies in $\mathrm{B}_1 \backslash \mathrm{B}_{\lambda}$: \begin{align*}
		        r |e| &\leq  \omega_{\lambda}^2, \\
		       r^3 | \Delta e | &\leq \omega_{\lambda}^4,\\
		       \diff^{\ast} e &= 0,
		\end{align*} then
		\begin{equation*} e = \sum_{1 \leq i < j \leq 4} (r^{- 4} a_{i j} + b_{i j}) \phi^{i j}
		+ \tilde{e} + \diff \eta \end{equation*}
		where $\phi^{i j} = (x^i \diff x^j - x^j \diff x^i) / 2$, for some $a_{i j}, b_{i j} \in \mathfrak{g}$, $\tilde{e}$ and $\eta$ such that \begin{itemize}
		   \item \begin{equation}
		       \eta =  \sum_{1 \leq j \leq 4}  \beta_j x^j +
	\frac12 \sum_{1 \leq i \leq j \leq 4} \nu_{i j} x^i x^j,
		   \end{equation}  \item in $\mathrm{B}_{1 / 2} \backslash \mathrm{B}_{2
			\lambda}$
		\begin{equation} r|\tilde{e}| +  r^2 | \diff \tilde{e} | \leq C \omega_{\lambda}^\alpha \label{key1}, \end{equation} 
		
		\item  \begin{equation}
		   \sum_{1 \leq j \leq 4}  r|\beta_j| + 
	 \sum_{1 \leq i \leq j \leq 4} r^2 |\nu_{i j}| +  \sum_{1 \leq i < j \leq 4}  r^{-2}|a_{i j}| + r^2|b_{i j}| \leq C  \omega_\lambda^2 \label{key2}.
		\end{equation}
		\end{itemize}
	\end{proposition}
	
	\begin{proof} The first two conditions imply that \cref{ellipticPDEannulus} and \cref{rkellipticPDEannulus} apply to the coefficients of $e$ in a coordinate frame, namely we have $e = \zeta + \vartheta$
		where, in $\mathrm{B}_{1 / 2} \backslash \mathrm{B}_{2 \lambda}$,
		\begin{equation} r | \zeta | + r^2 | \nabla \zeta | \leq C \omega_{\lambda}^\alpha \label{zetaestimate} \end{equation}
		and $\vartheta$ is harmonic of degree less than $1$ that is $\vartheta$ is of the form
		\begin{equation} \vartheta = \sum_{1 \leq j \leq 4} (r^{- 2} \alpha_j + \beta_j) \diff x^j +
		\sum_{1 \leq i \leq j \leq 4} (r^{- 4} \mu_{i j} + \nu_{i j}) \psi^{i
			j} + \sum_{1 \leq i < j \leq 4} (r^{- 4} a_{i j} + b_{i j}) \phi^{i j} \label{hformula}
		\end{equation}
		where $\phi^{i j} = (x^i \diff x^j - x^j \diff x^i) / 2$, $\psi^{i j} =
		(x^i \diff x^j + x^j \diff x^i) / 2$ and the $\alpha_j$'s, $\beta_j$'s, $\mu_{ij}$'s, $\nu_{ij}$'s, $a_{ij}$'s, $b_{ij}$'s are constant elements of $\mathfrak{g}$. 
		\
		
		\noindent Since $r |\vartheta| = r|e-\zeta|\leq C \omega_{\lambda}^2$ in $\mathrm{B}_{1 / 2} \backslash \mathrm{B}_{2 \lambda}$, we have a first bound on the coefficients involved in \cref{hformula}, namely for $r\in (1/2,2\lambda)$, \begin{equation}
		    \sum_{1 \leq j \leq 4} r(r^{- 2} |\alpha_j| + |\beta_j|) +
		\sum_{1 \leq i \leq j \leq 4} r^2(r^{- 4} |\mu_{i j}| + |\nu_{i j}|) + \sum_{1 \leq i < j \leq 4} r^2(r^{- 4} |a_{i j}| + |b_{i j}|) \leq C\omega_{\lambda}^2 .
		\end{equation} This is enough to prove \eqref{key2} but is not precise enough for \eqref{key1}. We will use the assumption $\diff^*e$ to partly tighten the previous estimate. 
		\
		
		\noindent Notice that $\vartheta$ can be written in the form
		\begin{equation*} \vartheta = \diff f - \left( \sum_{1 \leq j \leq 4} \alpha_j x^j \right) \diff
		r^{- 2} -\frac{1}{2} \left( \sum_{1 \leq i \leq j \leq 4} \mu_{i j} x^i x^j
	 \right) \diff r^{- 4} + \sum_{1 \leq i < j \leq 4}
		(r^{- 4} a_{i j} + b_{i j}) \phi^{i j} \end{equation*}
		where $f = \sum_{1 \leq j \leq 4} (r^{- 2} \alpha_j + \beta_j) x^j +
		\frac{1}{2} \sum_{1 \leq i \leq j \leq 4} (r^{- 4} \mu_{i j} + \nu_{i j})
		x^i x^j$ is harmonic. Then $v \coloneq \vartheta (r \partial_{r})$ satisfies
		\begin{equation}  v = r \partial_{r} f + 2 r^{- 2} 
		\sum_{1 \leq j \leq 4} \alpha_j x^j + 2 r^{- 4}\sum_{1 \leq i \leq j \leq 4}
		\mu_{i j} x^i x^j. \label{vformula} \end{equation}
		Since $f$ is harmonic then so is $r \partial_{r} f$ and therefore
		\begin{equation} \Delta v = - \frac{1}{r^3} \frac{\partial}{\partial r} \left(
		r^3 \frac{\partial}{\partial r} v \right) + \frac{1}{r^2}
		\Delta_{\mathbf{S}^3} v = 8 r^{- 3}  \sum_{1 \leq j \leq 4}
		\alpha_j \frac{x^j}{r} + 16 r^{- 4} \sum_{1 \leq i \leq j \leq 4}
		\mu_{i j} \frac{x^i x^j }{r^{2}}  \label{Dvformula} \end{equation}
		Additionally 
		\begin{equation} \Delta v = (\Delta \vartheta)(r \partial_r)+ 2 \diff^{\ast} \vartheta =2 \diff^{\ast}(e -\zeta)= - 2 \diff^{\ast} \zeta \end{equation}
		and by \cref{zetaestimate} we obtain \begin{equation}
		     | \Delta v | \leq r^{- 2} \omega_{\lambda}^\alpha. \label{Dvest}
		\end{equation} Notice that the
		functions $((x^j)_{1 \leq j \leq 4}, (x^i x^j)_{1 \leq i \leq j \leq 4})$,
		defined on $\mathbf{S}^3$, are linearly independent. The combination of \cref{Dvformula,Dvest} implies that for
		$r \in (2 \lambda, 1 / 2)$
		\begin{equation} r^{- 1} | \alpha_j | + r^{- 2} | \mu_{i j} | \leq C \omega_{\lambda}^\alpha  \label{vcoeffest}. \end{equation}
		Finally, if $\tilde{e} = \zeta + \sum_{1 \leq j \leq 4} r^{- 2} \alpha_j \diff x^j + \sum_{1 \leq i \leq j \leq 4} r^{- 4} \mu_{i j} \psi^{i j}$ then
		\begin{equation} \diff \tilde{e} = \diff \zeta - \sum_{1 \leq j \leq 4} \alpha_j \diff x^j
		\wedge \diff r^{- 2} -  \sum_{1 \leq i \leq j \leq 4} \mu_{i j}
		\psi^{i j} \wedge \diff r^{- 4} \end{equation}
		and
		\begin{equation*} r|\tilde{e}|+ r^2 | \diff \tilde{e} | \leq   r|\zeta| + r^2 | \diff \zeta | +C \sum_{1 \leq
			j \leq 4} | \alpha_j | r^{- 1} +C \sum_{1 \leq i \leq j \leq 4} |
		\mu_{i j} | r^{- 2} \leq C \omega_{\lambda}^\alpha \end{equation*} which is precisely estimate \eqref{key1}. 

	\end{proof}
	
	\begin{proof}[Proof of \cref{asymptoticneckexpansion}]
	Define the error connection $1$-form $e_k$ by
		\begin{equation*} e_k = A_k {- A_{\infty}}  - \lambda_k^{-1} \hat{A}_{\infty}(\lambda_k^{-1} \cdot).
		\end{equation*} From the Yang-Mills equation, of the form $\Delta A = A \cdot \nabla A
		+ A^3$, we obtain
		\begin{equation*} r^3 | \Delta A_k | \leq C \omega_{\lambda_k}^4, \end{equation*}
		where the operator on the right is at first the Laplace-Beltrami operator for the metric $h$ but can then be replaced by the flat Laplacian on $\R^4$, since their difference is controlled thanks to \cref{flatornot}. From
		this we deduce that similar estimates hold for $e_k$ that is
		\begin{align*} \diff^{\ast} e_k &= 0,\\
		r | e_k | + r^2 | \nabla e_k | &\leq C	\omega_{\lambda_k}^2,\\
	 r^3 | \Delta e_k | &\leq C \omega_{\lambda_k}^4 . \end{align*}
		Now \cref{key} applies to $e_k$: there exists 2-forms $\Omega_k$ and $\hat{\Omega}_k$, with
		constant coefficients bounded in $k$, such that
		\begin{equation*} r^2 | \diff e_k - \Omega_k - \lambda_k^2 \psi^\ast\hat{\Omega}_k | \leq C
		\omega_{\lambda_k}^\alpha \end{equation*} for any $\alpha\in (2,3)$,  where $\Omega_k$ and $\hat{\Omega}_k $ are 2-forms on $\mathbf{R}^4$ with constant coefficients and $\psi(x)=x/|x|^2$. Notice that $r^2 | \diff e_k
		|$ converges to $0$ uniformly on annuli of the form $\mathrm{B}_1 \backslash
		\mathrm{B}_{\delta}$ and $\mathrm{B}_{\lambda_k / \delta} \backslash
		\mathrm{B}_{\lambda_k}$ for all $\delta \in (0, 1)$. We deduce for $x\in \mathrm{B}_1 \backslash
		\mathrm{B}_{\delta}$ \begin{equation*} \limsup_{k\to\infty }r^2 |\Omega_k| \leq \lim_{k\to\infty} \left( C(r+\lambda_k r^{-1})^\alpha + r^2 |\diff e_k | + \frac{\lambda_k^2}{r^2} |\hat{\Omega}_k |\right) = C r^{\alpha} \end{equation*} and in particular \begin{equation*} \limsup_{k\to\infty } |\Omega_k| \leq C \delta^{\alpha-2}\end{equation*} which implies $\Omega_k\to 0$ by letting $\delta$ go to $0$, since $\alpha >2$. Similarly, $\hat{\Omega}_k\to 0$. This proves that \begin{equation*}\sup_{\partial
			\mathrm{B}_{\sqrt{\lambda_k}}}|\diff e_k| \leq \lambda_k^{-1} \lambda_k^{\alpha/2} +  |\Omega_k| + |\hat{\Omega}_k| 
			\to 0\end{equation*}  Finally, since $R_k \coloneq 	F_{A_k} - F_{A_{\infty}} - \lambda_k^{-2} F_{\hat{A}_{\infty}}(\lambda_k^{-1} \cdot)$ satisfies \begin{equation*}
			     R_k =  \diff e_k + [A_{\infty}, \lambda_k^{-1}
			\hat{A}_{\infty}(\lambda_k^{-1} \cdot)] + [e_k, A_{\infty} + \lambda_k^{-1}
			\hat{A}_{\infty}(\lambda_k^{-1} \cdot)] + e_k \wedge e_k
			\end{equation*} we conclude that $R_k$ is uniformly converging to $0$ at scale $\partial
		\mathrm{B}_{\sqrt{\lambda_k}}$.
	\end{proof}
	\begin{remark}
	    \label{estimategeneralbubbletree} Notice that the estimate of \cref{asymptoticneckexpansion} still holds in the general bubble-tree case: the proof remains valid because by construction other bubbles cannot be seen at the boundary of the neck regions.
	\end{remark}

\section{Proof of the main theorem} \label{proofmainthm}
	\subsection{Infinitesimal conformal symmetry}
		
		\begin{theorem}\label{mainthm2}
		   Consider a sequence $(A_k)_k$ as in \cref{bubbletree} with bubble-tree limit $(A_\infty, \widehat{A}_\infty)$. Denote by $F^\infty$, $\widehat{F}^\infty$ and $\tilde{F}^\infty$ the curvature tensors of $A_\infty$, $\widehat{A}_\infty$ and $\psi^*\widehat{A}_\infty$ respectively, where $\psi:x \mapsto x/|x|^2$ is the inversion. Consider $\gamma$ the degree 2-term of the Taylor expansion of $h$ in geodesic coordinates at $0$. Then the $2$-tensor $\mathrm{P}$ defined below is traceless symmetric:
		    \begin{equation} \label{obstructionbubbling} \mathrm{P} = 
		    \left( F^{\infty} \circ\tilde{F}^{\infty} \right)(0) -\frac{1}{\pi^2} \int_{\R^4} \left(\frac{1}{2} \langle  S_{\widehat{F}^{\infty}} , \nabla\gamma\rangle \otimes r \diff r -S_{\widehat{F}^{\infty}} \circ \gamma + \frac{1}{4}\langle S_{\widehat{F}^{\infty}}, \gamma \rangle \xi \right)\mathrm{vol}.
		    \end{equation}
		\end{theorem}
	
	\begin{proof} The theorem can be reformulated as follows: $\mathrm{P}$ is orthogonal to $\mathfrak{conf}(4)$, that is vanishes modulo $\mathfrak{conf}(4)$. The following computations will therefore be conducted modulo $\mathfrak{conf}(4)$ and equality in the quotient space will be denoted by $\equiv$. Note that, since the orthogonal projection on $\mathfrak{conf}(4)$ is continuous, $\equiv$ is preserved by taking limits. 
	
	\noindent We recall here the result of \cref{simplifypohozaevtensor}: \begin{align}
   0 \equiv &\int_{\Sph^3} \left(\iota_{\partial_r}S_{F_{A_k}}(\sqrt{\lambda_k}\cdot) \otimes r\diff r\right) \mathrm{vol}_{\Sph^3}\notag\\ 
    &- \int_{\mathrm{B}_{1/\delta}} \left(\frac{1}{2} \langle  S_{F_{\widehat{A}_k}} , \nabla\gamma\rangle \otimes r \diff r -S_{F_{\widehat{A}_k}} \circ \gamma + \frac{1}{4}\langle S_{F_{\widehat{A}_k}}, \gamma \rangle \xi \right)\mathrm{vol} + \varpi_{k,\delta} \label{obstruction1}
\end{align}  where $\widehat{A}_k = \lambda_k^{-1}\widehat{A}_{\infty}(\lambda_k^{-1}\cdot)$, $\gamma$ is the degree 2-term of the Taylor expansion of $h$ in geodesic coordinates at $0$ and \begin{equation}
    \lim_{\delta\to 0}\lim_{k\to +\infty} \varpi_{k,\delta} = 0.
\end{equation}
Let's first deal with the boundary term. Note that if $\Phi$ is a $2$-tensor on $\R^4 = T_0 \R^4$, then \begin{align*}
    \int_{\Sph^3}\left( \iota_{\partial_r} \Phi \otimes \diff r \right)\mathrm{vol}_{\Sph^3} &=  \int_{\Sph^3} \left(x^i \Phi_{ij} \diff x^j   \otimes x^k \diff x^k\right)\mathrm{vol}_{\Sph^3} \\
    &= \Phi_{ij} \left(\int_{\Sph^3} x^i x^k \mathrm{vol}_{\Sph^3}\right)\diff x^j   \otimes\diff x^k\\
    &= \Phi_{ij} \delta^{ik} \frac{|\Sph^3|}{4} \diff x^j   \otimes\diff x^k\\
    &= \frac{\pi^2}{2} \Phi_{ij}  \diff x^j   \otimes\diff x^i 
\end{align*} If $\Phi$ is now a tensor field, then from the previous computation, we deduce \begin{equation}
     \lim_{k\to +\infty}  \int_{\Sph^3} \left(\iota_{\partial_r} \Phi(\sqrt{\lambda_k}\cdot) \otimes \diff r\right) \mathrm{vol}_{\Sph^3} = \frac{\pi^2}{2} \Phi_{ij}(0) \diff x^j   \otimes\diff x^i \label{bdytermcst}
\end{equation} 
Using the fact that the stress-energy tensor is quadratic in the curvature, we can write \begin{equation}
    S_{F_{A_k}} = S_{F_\infty} + S_{F_{\widehat{A}_k}} + \frac{1}{2} \langle F^{\infty}, F_{\widehat{A}_k} \rangle \xi -  F^{\infty}\circ F_{\widehat{A}_k} -   F_{\widehat{A}_k}\circ F^{\infty} +   \mathcal{R}_k
\end{equation} where according to \cref{asymptoticneckexpansion}: \begin{equation}\label{reste}
    \sup_{\partial \mathrm{B}_{\sqrt{\lambda_k}}} |\mathcal{R}_k|\to 0.
\end{equation} Denote by $\sigma_k: x \mapsto \lambda_k x/|x|^2$. Then $\widehat{A}_k= \sigma_k^* \psi^* \widehat{A}_{\infty}$ and \begin{equation}
     S_{F_{A_k}} = S_{F_\infty} + S_{\sigma_k^*\tilde{F}^{\infty}} + \frac{1}{2} \langle F^{\infty}, \sigma_k^*\tilde{F}^{\infty} \rangle \xi -  F^{\infty}\circ \sigma_k^*\tilde{F}^{\infty} -   \sigma_k^*\tilde{F}^{\infty}\circ F^{\infty} +   \mathcal{R}_k \label{expansionS}
\end{equation} Since $\sigma_k $ is nothing but the inversion with respect to the sphere $\partial\mathrm{B}_{\lambda_k}$, we have for $x\in \partial\mathrm{B}_{\lambda_k}$, $\diff_x \sigma_k = I - 2 \diff r\otimes \partial_r$ and thus \begin{equation}
    \sigma_k^*\tilde{F}^{\infty} = \tilde{F}^{\infty} - 2 \diff r\wedge \iota_{\partial_r} \tilde{F}^{\infty}. \label{sigma1}
\end{equation} On the one hand \begin{align}
    \left(F^{\infty}\circ(\diff r\wedge \iota_{\partial_r}\tilde{F}^{\infty})\right) (\partial_r, X) &= \langle \iota_{\partial_r} F^{\infty}, \iota_X \left(\diff r\wedge \iota_{\partial_r}\tilde{F}^{\infty}\right)\rangle \notag\\
    &=\langle \iota_{\partial_r} F^{\infty}, \iota_{\partial_r} \tilde{F}^{\infty}\rangle\diff r(X) - \langle \iota_{\partial_r} F^{\infty}, \diff r\wedge \iota_X\iota_{\partial_r}\tilde{F}^{\infty}\rangle \notag\\
    &=\langle \iota_{\partial_r} F^{\infty}, \iota_{\partial_r} \tilde{F}^{\infty}\rangle\iota_{\partial_r}\xi(X)- \langle \iota_{\partial_r}\iota_{\partial_r} F^{\infty},  \iota_X\iota_{\partial_r}\tilde{F}^{\infty}\rangle\notag\\
    &=\langle \iota_{\partial_r} F^{\infty}, \iota_{\partial_r} \tilde{F}^{\infty}\rangle\iota_{\partial_r}\xi(X) \label{sigma2}
\end{align} and on the other hand \begin{align}
    \left((\diff r\wedge \iota_{\partial_r}\tilde{F}^{\infty}\circ F^{\infty}\right) (\partial_r, X) &= \langle \iota_{X} F^{\infty}, \iota_{\partial_r} \left(\diff r\wedge \iota_{\partial_r}\tilde{F}^{\infty}\right)\rangle \notag\\
    &=\langle \iota_{X} F^{\infty}, \iota_{\partial_r}\tilde{F}^{\infty}\rangle  - \langle \iota_{X} F^{\infty}, \diff r\wedge \iota_{\partial_r}\iota_{\partial_r}\tilde{F}^{\infty}\rangle \notag\\
    &=\left(\tilde{F}^{\infty}\circ F^{\infty} \right) (\partial_r, X) \label{sigma3}
\end{align}

Therefore, plugging \cref{sigma1,sigma2,sigma3} in \cref{expansionS} we get \begin{align*}
    \iota_{\partial_r}S_{F_{A_k}} &=  \iota_{\partial_r} \left( S_{F_\infty} + S_{\sigma_k^*\tilde{F}^{\infty}} + \mathcal{R}_k + \frac{1}{2} \langle F^{\infty}, \tilde{F}^{\infty} \rangle \xi -  F^{\infty}\circ \tilde{F}^{\infty} -   \tilde{F}^{\infty}\circ F^{\infty} \right) \\
     &\quad +\iota_{\partial_r}\left(- \langle  F^{\infty},\diff r\wedge \iota_{\partial_r}\tilde{F}^{\infty}\rangle\xi +2 F^{\infty}\circ(\diff r\wedge \iota_{\partial_r}\tilde{F}^{\infty}) + 2  (\diff r\wedge \iota_{\partial_r}\tilde{F}^{\infty})\circ F^{\infty} \right) \\
    &=  \iota_{\partial_r} \left( S_{F_\infty} + S_{\sigma_k^*\tilde{F}^{\infty}} + \mathcal{R}_k + \frac{1}{2} \langle F^{\infty}, \tilde{F}^{\infty} \rangle \xi -  F^{\infty}\circ \tilde{F}^{\infty} -   \tilde{F}^{\infty}\circ F^{\infty} \right) \\
     &\quad  +\iota_{\partial_r}\left(- \langle  F^{\infty},\diff r\wedge \iota_{\partial_r}\tilde{F}^{\infty}\rangle\xi +2\langle\iota_{\partial_r} F^{\infty}, \iota_{\partial_r} \tilde{F}^{\infty}\rangle \xi +2 \tilde{F}^{\infty}\circ F^{\infty} \right)\\
    &=  \iota_{\partial_r} \left( S_{F_\infty} + S_{\sigma_k^*\tilde{F}^{\infty}} + \mathcal{R}_k + \frac{1}{2} \langle F^{\infty}, \tilde{F}^{\infty} \rangle \xi -  F^{\infty}\circ \tilde{F}^{\infty} +   \tilde{F}^{\infty}\circ F^{\infty} \right)
\end{align*}  where in the last step, we have used that, with our normalisation for the inner product of tensors, $\langle  F^{\infty},\diff r\wedge \iota_{\partial_r}\tilde{F}^{\infty}\rangle = 2\langle  \iota_{\partial_r} F^{\infty}, \iota_{\partial_r}\tilde{F}^{\infty}\rangle.$ Finally, we get the expression \begin{align}
    \iota_{\partial_r}S_{F_{A_k}} =  \iota_{\partial_r} \left( S_{F_\infty+  \tilde{F}^{\infty}} - S_{\tilde{F}^{\infty}} + S_{\sigma_k^*\tilde{F}^{\infty}} + \mathcal{R}_k +2 \tilde{F}^{\infty}\circ F^{\infty} \right) \label{aux1}
\end{align}
From \cref{bdytermcst}, \begin{align}
    \lim_{k\to +\infty} \frac{2}{\pi^2}\int_{\Sph^3}\left( \iota_{\partial_r} \left( S_{F_\infty+  \tilde{F}^{\infty}} - S_{\tilde{F}^{\infty}} + 2 \tilde{F}^{\infty}\circ F^{\infty} \right)\right.&\left.(\sqrt{\lambda_k}\cdot) \otimes \diff r \right)\mathrm{vol}_{\Sph^3} \notag\\ &= \left(S_{F_\infty+  \tilde{F}^{\infty}} - S_{\tilde{F}^{\infty}} + 2 F^{\infty} \circ \tilde{F}^{\infty}\right)(0) \notag\\
    &\equiv 2 F^{\infty} \circ \tilde{F}^{\infty}(0) \label{aux2}
\end{align} where we have used that, by definition, stress-energy tensors are traceless symmetric. Moreover, using \cref{reste} \begin{equation}
     \lim_{k\to +\infty} \int_{\Sph^3} \left(\iota_{\partial_r} \mathcal{R}_k (\sqrt{\lambda_k}\cdot) \otimes \diff r\right) \mathrm{vol}_{\Sph^3} = 0 \label{aux3}
\end{equation} and since $\sigma_k$ is a conformal transformation, $S_{\sigma_k^*\tilde{F}^{\infty}}$ is the stress-energy tensor of a Yang-Mills connections hence satisfies \cref{deftensorpohozaev} with $h=\xi$, that is  \begin{equation}
    \int_{\Sph^3} \left(\iota_{\partial_r} S_{\sigma_k^*\tilde{F}^{\infty}} (\sqrt{\lambda_k}\cdot) \otimes \diff r \right) \mathrm{vol}_{\Sph^3} \equiv 0 \label{aux4}
\end{equation} Combining \cref{obstruction1} with \cref{aux1,aux2,aux3,aux4}, we get \begin{equation}
   0 \equiv  \pi^2 F^{\infty} \circ \tilde{F}^{\infty}(0) - \int_{\mathrm{B}_{1/\delta}} \left(\frac{1}{2} \langle  S_{F_{\widehat{A}_k}} , \nabla\gamma\rangle \otimes r \diff r -S_{F_{\widehat{A}_k}} \circ \gamma + \frac{1}{4}\langle S_{F_{\widehat{A}_k}}, \gamma \rangle \xi \right)\mathrm{vol} + \widetilde{\varpi}_{k,\delta}\label{aux5}
\end{equation} where \begin{equation*}
    \lim_{\delta\to 0}\lim_{k\to +\infty} \widetilde{\varpi}_{k,\delta} = 0.
\end{equation*} The theorem follows from the smooth convergence of $\widehat{A}_k$ to $\widehat{A}_{\infty}$.
	\end{proof}	
	
	To complete the proof of \cref{mainthm}, the only thing left is the following result.
	\begin{proposition}\label{mainthm3} In the conclusion of \cref{mainthm2}, $\gamma$ can be replaced by $ x\mapsto -\frac{1}{3}\mathrm{W}(\cdot,x,\cdot,x)$ where $\mathrm{W}$ is the Weyl tensor at $0$. This can be reformulated as follows: for every pair of indices $(i,j)$
		    \begin{equation} \label{mainthmeq}
		         \langle F^{\infty}_{\alpha[i} , \tilde{F}^{\infty}_{j]\alpha} \rangle (0) + \delta_{ij} \langle F^{\infty}, \tilde{F}^{\infty}\rangle (0) +  \int_{\R^4} \langle  S_{\widehat{F}^{\infty}} \otimes \mathrm{W},    \mathrm{T}^{i j} \rangle \mathrm{vol} = 0
		    \end{equation} where \begin{equation*}
		        \mathrm{T}^{i j}_{ab\alpha\mu\beta\nu} = \frac{1}{3\pi^2}\left(\delta_{\alpha a}  \delta_{\beta b} x^\nu x^{[j}\delta_{i]\mu} -\delta_{\alpha a} \delta_{b[i}\delta_{j]\beta} x^\mu x^\nu   +\delta_{\alpha a}  \delta_{\beta b}\delta_{ij}  x^\mu x^\nu  \right)
		    \end{equation*}
		\end{proposition}
	
	\begin{proof} In geodesic coordinates, one has the following expression for $\gamma$ \cite[proposition II.3.1]{sakai1996riemannian}
	: \begin{equation}
	    \gamma(x) = -\frac{1}{3} \mathrm{Rm}(\cdot,x,\cdot,x) 
	\end{equation} $\mathrm{Rm}$ is the Riemann tensor at $0$. Moreover, using the Ricci decomposition of the Riemann tensor, it is enough to prove that, defining $S=S_{\widehat{F}_\infty}$, the 2-tensor \begin{equation*} \Phi \coloneq
 \int_{\R^4} \left(\frac{1}{2} \langle  S , \nabla\sigma\rangle \otimes r \diff r -S \circ \sigma + \frac{1}{4}\langle S, \sigma \rangle \xi \right)\mathrm{vol} 
	\end{equation*} is traceless symmetric, for every $\sigma$ of the form $\sigma(x) = (R\owedge\xi)(\cdot,x,\cdot, x)$ where $R$ is an arbitrary symmetric $2$-tensor and $\owedge$ is the Kulkarni-Nomizu product. First,  \begin{equation*}
	    \sigma =f\xi+ \nabla^{\mathrm{s}} \omega
	\end{equation*} where: \begin{align*}
	    f(x) &= 3 R(x,x), \\
	    \omega(x) &= |x|^2 R(x,\cdot) -2R(x,x) \xi(x,\cdot)
	\end{align*} and $\nabla^{\mathrm{s}}$ denote the symmetric gradient: $(\nabla^{\mathrm{s}}\omega)_{\alpha\beta} = \left(\partial_\alpha \omega_\beta + \partial_\beta \omega_\alpha \right)/2$. Since $S$ is traceless, we have in particular \begin{equation} \label{curvtermtrace}
	    \mathrm{Tr} \,\Phi =  \int_{\R^4} \frac{1}{2} \mathrm{Tr}\left(\langle  S , \nabla\sigma\rangle \otimes  r \diff r\right) \mathrm{vol}  = \int_{\R^4} \frac{1}{2} \langle  S , r\partial_r\sigma\rangle   \mathrm{vol} = \int_{\R^4}  \langle  S , \sigma\rangle   \mathrm{vol} = 0
	\end{equation} where we have used that the coefficients of $\sigma$ are homogeneous of degree $2$. Additionally, \begin{align*}
	      \nabla \sigma \otimes r\diff r &= \left(\nabla f\otimes r\diff r\right)\xi +  \nabla \left(\nabla^{\mathrm{s}}\omega\right)\otimes r\diff r\\
	         S\circ \sigma &= f S\circ \xi + S\circ \nabla^{\mathrm{s}} \omega =    f S + S\circ \nabla^{\mathrm{s}} \omega
	\end{align*} and since $S$ is traceless we deduce \begin{equation*}
	   \Phi = \int_{\R^4} \left(\frac12\langle S,V \rangle- f S \right)  \mathrm{vol}
	\end{equation*} where $\langle S,V \rangle$ is understood as \begin{equation*}
	    \langle S,V \rangle = S_{\alpha\beta} V_{ij}^{\alpha\beta} \diff x^i\otimes \diff x^j
	\end{equation*} with \begin{equation*}
	  V_{ij}^{\alpha\beta} \coloneq x^j\partial_i \left(\nabla^{\mathrm{s}}\omega\right)_{\alpha\beta} - \delta_{i\beta}\nabla^{\mathrm{s}}_{\alpha j}\omega - \delta_{i\alpha}\nabla^{\mathrm{s}}_{\beta j}\omega.
	\end{equation*} On the one hand \begin{equation*}
	    x^j\partial_i \left(\nabla^{\mathrm{s}}\omega\right)_{\alpha\beta} = \left(\nabla^{\mathrm{s}}( x^j\partial_i\omega)\right)_{\alpha\beta} -\frac12 \delta_{\alpha j}\partial_i\omega_\beta - \frac12 \delta_{\beta j}\partial_i\omega_\alpha
	\end{equation*} and on the other hand \begin{align*}
	  2 \delta_{i\beta}\nabla^{\mathrm{s}}_{\alpha j}\omega +2 \delta_{i\alpha}\nabla^{\mathrm{s}}_{\beta j}\omega
	    &= \delta_{i\beta} \partial_\alpha \omega_j +\delta_{i\beta}\partial_j \omega_\alpha + \delta_{i\alpha} \partial_\beta \omega_j +\delta_{i\alpha}\partial_j \omega_\beta \\
	    &=\partial_\alpha( \delta_{i\beta}  \omega_j) +\partial_\beta(\delta_{i\alpha}  \omega_j) +  \delta_{i\beta}\partial_j \omega_\alpha + \delta_{i\alpha}\partial_j \omega_\beta\\
	    &= 2\nabla^{\mathrm{s}}_{\alpha\beta}(\omega_j\diff x^i) + \delta_{i\beta}\partial_j \omega_\alpha + \delta_{i\alpha}\partial_j \omega_\beta
	\end{align*} so we deduce \begin{equation*}
	      V_{ij}^{\alpha\beta} = \left(\nabla^{\mathrm{s}}( x^j\partial_i\omega + \omega_j\diff x^i)\right)_{\alpha\beta} - \frac12\left( \delta_{\alpha j}\partial_i\omega_\beta + \delta_{\beta j}\partial_i\omega_\alpha + \delta_{i\beta}\partial_j \omega_\alpha + \delta_{i\alpha}\partial_j\omega_\beta \right)
	\end{equation*} that is \begin{equation*}
	    \langle S, V\rangle =    \langle S, \nabla^{\mathrm{s}}\left(\nabla \omega\otimes r\diff r \right)\rangle  - \left( S_{j\beta}\partial_i \omega_\beta + S_{\alpha i}\partial_j \omega_\alpha\right) \diff x^i\otimes\diff x^j
	\end{equation*} and finally, since $S$ is divergence-free:
	\begin{equation*} \Phi =- \int_{\R^4} \left( f S + \left( S_{j\beta}\partial_i \omega_\beta + S_{\alpha i}\partial_j \omega_\alpha\right) \diff x^i\otimes\diff x^j\right)  \mathrm{vol}
	\end{equation*} which is indeed a symmetric tensor. Combining this with \cref{curvtermtrace}, $\Phi$ is traceless symmetric. Choosing $R$ such that $\mathrm{Rm}= \mathrm{W} + R\owedge\xi$, we deduce that \begin{equation*}
	    \mathrm{P} = 
		    \left( F^{\infty} \circ\tilde{F}^{\infty} \right)(0) +\frac{1}{3\pi^2} \int_{\R^4} \left(\frac{1}{2} \langle  S , \nabla w \rangle \otimes r \diff r -S \circ w + \frac{1}{4}\langle S, w \rangle \xi \right)\mathrm{vol}
	\end{equation*}  is traceless symmetric, where $w=\mathrm{W}(\cdot, x, \cdot, x)$. Equation \cref{mainthmeq} will follow from the computation of the trace and skew-symmetric part of $\mathrm{P}$. \begin{align*}
	   P_{[ij]} + \mathrm{Tr(P)}\delta_{ij} &=  \left( F^{\infty} \circ\tilde{F}^{\infty} \right)_{[ij]}(0) + \delta_{ij} \langle F^{\infty}, \tilde{F}^{\infty} \rangle(0) \\
	  & \quad +  \frac{1}{3\pi^2} \int_{\R^4} \left(\frac{1}{2} \langle  S , x^{[j} \partial_{i]}w -(S\circ w)_{[ij]} + \delta_{ij}\langle S, w \rangle \right)\mathrm{vol}\\
	   &= \delta_{ij} \left(\langle F^{\infty}, \tilde{F}^{\infty} \rangle(0) +  \frac{1}{3\pi^2} \mathrm{W}_{\alpha\mu\beta\nu}\int_{\R^4} S_{\alpha\beta} x^\mu x^\nu  \mathrm{vol}\right) \\
	   &\quad +\left( F^{\infty} \circ\tilde{F}^{\infty} \right)_{[ij]}(0)\\
	   &\quad +\frac{1}{3\pi^2} \int_{\R^4} \left(\frac{1}{2} \mathrm{W}_{\alpha\mu\beta\nu}S_{\alpha\beta} x^{[j} \partial_{i]}(x^\mu x^\nu) -x^\mu x^\nu S_{\alpha [i}\mathrm{W}_{j] \nu\alpha\mu } \right)\mathrm{vol} \\
	   &= \delta_{ij} \left(\langle F^{\infty}, \tilde{F}^{\infty} \rangle(0) +  \frac{1}{3\pi^2} \mathrm{W}_{\alpha\mu\beta\nu}\int_{\R^4} S_{\alpha\beta} x^\mu x^\nu  \mathrm{vol}\right) \\
	   &\quad +\left( F^{\infty} \circ\tilde{F}^{\infty} \right)_{[ij]}(0)\\
	   &\quad +\frac{1}{3\pi^2} \int_{\R^4} \left( \mathrm{W}_{\alpha\mu\beta\nu}S_{\alpha\beta}x^\nu x^{[j} \delta_{i]\mu} -x^\mu x^\nu S_{\alpha [i}\mathrm{W}_{j] \nu\alpha\mu } \right)\mathrm{vol}\\
	   &=  \delta_{ij}\langle F^{\infty}, \tilde{F}^{\infty} \rangle(0)+ \left( F^{\infty} \circ\tilde{F}^{\infty} \right)_{[ij]}(0) +\int_{\R^4} S_{a b} \mathrm{W}_{\alpha\mu\beta\nu} \mathrm{T}^{ij}_{ab\alpha\mu\beta\nu} \mathrm{vol}
	\end{align*} which concludes the proof.
	\end{proof} 
	\begin{remark}
	    Since the problem under study is invariant under conformal changes of the metric, we could have used conformal normal coordinates (see \cite{lee1987yamabe}) to get rid of the last proposition. It is interesting to note that our approach is independent on the choice of normal coordinates and still yields a conformally invariant result.
	\end{remark}

	\subsection{Infinitesimal gauge symmetry}
	\begin{theorem}
	\label{mainthm5}
	      Consider a sequence $(A_k)_k$ of Yang-Mills connections with weak limit $A_\infty$. Assume that $0\in M$ is a concentration point and that the rescaled sequence about $0$, denoted by $(\widehat{A}_k)_k$, weakly converges to  $\widehat{A}_\infty$. Denote by $F^\infty$ and $\tilde{F}^\infty$ the curvature tensors of $A_\infty$ and $\tilde{A}_\infty = \psi^*\widehat{A}_\infty$ respectively, where $\psi:x \mapsto x/|x|^2$ is the inversion. Then for every $\xi \in \mathfrak{g}$,
		    \begin{equation*}
		         \langle F^{\infty} ,  [\tilde{F}^{\infty}, \xi] \rangle (0) = 0. 
		         \end{equation*}
	\end{theorem}
	
\begin{proof} The proof follows the same line of reasoning as the one of \cref{mainthm2}, using this time the conservation law associated to gauge invariance that is for every $\xi \in \mathfrak{g}$, 
    \begin{equation}
      \frac{1}{\lambda_k ^2}  \int_{\partial\mathrm{B}_{\sqrt{\lambda_k}}} \langle \iota_\nu F_{A_k}, [A_k , \xi] \rangle_h {\mathrm{vol}_{h}}_{| \partial B} = 0.
    \end{equation}
    Using the same arguments, we deduce \begin{equation}
      \frac{1}{\lambda_k ^2}  \int_{\partial\mathrm{B}_{\sqrt{\lambda_k}}} \left(\langle \iota_\nu F^\infty, [\tilde{A}_\infty + e_k, \xi] \rangle - \langle \iota_\nu \tilde{F}^\infty, [A_\infty + e_k, \xi] \rangle \right)\mathrm{vol}  = o(1).
    \end{equation} Moreover, with our choice of gauge, $A_\infty(0)= \tilde{A}_\infty(0)=0$ and therefore \begin{align*}
        A_\infty &= x^i \partial_i  A_\infty(0) + O(r^2) \\
        \tilde{A}_\infty &=x^i \partial_i  \tilde{A}_\infty(0) + O(r^2) 
    \end{align*} which gives \begin{equation}
        \frac{1}{\lambda_k ^2}  \int_{\partial\mathrm{B}_{\sqrt{\lambda_k}}} \left(\langle \iota_\nu F^\infty, [\tilde{A}_\infty, \xi] \rangle - \langle \iota_\nu \tilde{F}^\infty, [A_\infty, \xi] \rangle \right)\mathrm{vol} \to \frac{|\Sph^3|}{4}  \langle F^{\infty} ,  [\tilde{F}^{\infty}, \xi] \rangle (0).
    \end{equation}  Recall that \cref{key} applies to $e_k$: we can write \begin{equation*}
        e_k = \check{e}_k + \sum_{1 \leq j \leq 4}  \beta_j^k \diff x^j +
	 \frac12 \sum_{1 \leq i\leq j \leq 4} \nu_{i j}^k \left(x^i \diff x^j +  x^j \diff x^i\right)
    \end{equation*} where $\sup_{\partial \mathrm{B}_{\sqrt{\lambda_k}}} |\check{e}_k| = o(\sqrt{\lambda}_k)$ and $|\beta_j^k|=O(\sqrt{\lambda_k})$, $|\nu^k_{ij}|=O(\lambda_k)$. Introducing $\Phi =F^\infty - \tilde{F}^\infty$, this implies \begin{align*}
        \frac{1}{\lambda_k ^2}  \int_{\partial\mathrm{B}_{\sqrt{\lambda_k}}} \langle \iota_\nu  \Phi , [e_k, \xi] \rangle \mathrm{vol} &=  \frac12\sum_{1 \leq i\leq j \leq 4} \left\langle [\xi , \nu_{i j}^k], \int_{\Sph^3} \left( x^\ell x^i \Phi_{\ell j}(\sqrt{\lambda_k}\cdot) +x^\ell x^j \Phi_{\ell i}(\sqrt{\lambda_k}\cdot) \right) \mathrm{vol} \right\rangle\\
         &\quad +\sum_{1 \leq j \leq 4}  \left\langle [\xi, \beta_j^k], \int_{\Sph^3} x^\ell \Phi_{\ell j}(\sqrt{\lambda_k}\cdot) \mathrm{vol} \right\rangle +o(1).
    \end{align*} All integral terms vanish in the limit, which is still true when they are multiplied by bounded quantities. As a result, the left hand side converges to zero and the proof is complete.
\end{proof}

		\subsection{The general bubble-tree case} 		
			To conclude this section, we extend the previous results of this section to the general bubble-tree setting, that we recall briefly (we refer to the paper of Laurain-Rivière \cite{laurainriviere2014angular} were the construction is performed in  conformally invariant problems in dimension 2, for Yang-Mills, see \cite[Theorem VII.3]{rivière2015variations}, \cite{gauvritlaurain2025morse}). The main idea is the following, as a consequence of $\epsilon$-regularity and energy gap, if the energy concentrates then one can progressively extract bubbles and the process must stop since each step takes a quantified amount of energy.
			
			 \begin{theorem}
	\label{bubbletreethm}
		Let $ (M^4,h)$ be a closed four-dimensional Riemannian manifold and  $(A_k)_k$ a sequence of Yang-Mills connection with uniformly bounded energy. Then we have, up to a sub-sequence, 		\begin{enumerate} 	
			\item There exists finitely many points $\{p_1,\dots, p_N\}$ and a Yang-Mills connection $A^{0,1}$ on $M$ such that $A_k$ converges to $A^{0,1}$ locally smootly modulo gauge, away from $\{p_1,\dots, p_N\}$,
			\item For each $i\in \{1, \dots, N\}$, there exists $N_i\in \N$ sequences of points $(p_k^{i,j})_k$ converging to $p_i$,  $N_i$ sequences of scalars $(\lambda_k^{i,j})_k$ converging to zero, $N_i$ non-trivial Yang-Mills connections on $\Sph^4$, called bubbles,  $A^{i,1},\dots A^{i,N_i}$ such that $(\phi_{k}^{i,j})^*(A_k)$ converges to $\pi_*A^{i,j} $ locally smootly modulo gauge, away from finitely many points, where $\phi_{k}^{i,j} (x)= p_k^{i,j} +\lambda _k^{i,j} x$ in local coordinates and $\pi$ is the stereographic projection. 
			
			 We call $\mathcal{B}:=\{A^{i,j}  |  i \in \llbracket 0,N\rrbracket, j \in \llbracket 1,N_i\rrbracket  \}$	the bubble-tree limit of the sequence, with the convention $N_0=1$.	\item Moreover there is no loss of energy, i.e 			$$ \lim_{k\rightarrow +\infty}\int_{M^4} \vert F_{A_k}\vert^2_h\, \mathrm{vol}_h  = \int_{M^4} \vert F_{A^{0,1}}\vert^2_h\, \mathrm{vol}_h +\sum_{i=1}^N \sum_{j=1}^{N_i} \int_{\Sph^4} \vert F_{A^{i,j}}\vert^2\, \mathrm{vol}. $$	\end{enumerate} 
	\end{theorem} 
	\begin{definition}
		The bubble-tree $\mathcal{B}$ as defined in \cref{branchthm} can be ordered in the following manner: for all $i\in \llbracket 1,N\rrbracket, j \in \llbracket 1,N_i\rrbracket$, $A^{0,1} \ll A^{i,j}$ and for all $i\in \llbracket 1,N\rrbracket, j , l \in \llbracket 1,N_i\rrbracket$
		${A}^{i,j} \ll {A}^{i,l}$ if $\lambda_k^{i,l} =o(\lambda_k^{i,j})$. Moreover, if $(i,j)$ and $(i',j')$, $i'\in\{0,i\}$, are such that ${A}^{i',j'} \ll {A}^{i,j}$ and there is no $l$ satisfying  ${A}^{i',j'} \ll {A}^{i,l} \ll {A}^{i,j}$, we write  ${A}^{i',j'} < {A}^{i,j}$.
	\end{definition}
	\begin{remark}
	\label{branch}
This order relation permits to define branches for each point of concentration $p_i$ there exists $k_i$ branches which consist of a bottom bubble $\hat{A}^{i,l}$ and a  set of $\{\hat{A}^{i,j,l} \, \vert \, l\in \llbracket 1,k_{i,j}\rrbracket \}$ such that for all $l\in \llbracket 1,k_{i,j}\rrbracket$ we have $\hat{A}^{i,j} \ll \hat{A}^{i,j,l}$. 
	\end{remark}
	  \begin{theorem} \label{mainthm4}
	      Consider a sequence $(A_k)_k$ of Yang-Mills connections with bubble-tree limit $\mathcal{B}$. Fix $A_\infty, \widehat{A}_\infty\in \mathcal{B}$ such that $A_\infty< \widehat{A}_\infty$. Identifying the concentration point with $0\in \R^4$ and denoting by $F^\infty$, $\widehat{F}^\infty$ and $\tilde{F}^\infty$ the curvature tensors of $A_\infty$, $\widehat{A}_\infty$ and $\psi^*\widehat{A}_\infty$ respectively, where $\psi:x \mapsto x/|x|^2$ is the inversion, then for every pair of indices $(i,j)$
		    \begin{equation*}
		         \langle F^{\infty}_{\alpha[i} , \tilde{F}^{\infty}_{j]\alpha} \rangle (0) + \delta_{ij} \langle F^{\infty}, \tilde{F}^{\infty}\rangle (0) +  \int_{\R^4} \langle  S_{\widehat{F}^{\infty}} \otimes \mathrm{W}, \mathrm{T}^{i j} \rangle \mathrm{vol} = 0
		    \end{equation*} where $\mathrm{W}$ is the Weyl tensor at $0$ in normal coordinates and \begin{equation*}
		        \mathrm{T}^{i j}_{ab\alpha\mu\beta\nu} = \frac{1}{3\pi^2}\left(\delta_{\alpha a}  \delta_{\beta b} x^\nu x^{[j}\delta_{i]\mu} -\delta_{\alpha a} \delta_{b[i}\delta_{j]\beta} x^\mu x^\nu   +\delta_{\alpha a}  \delta_{\beta b}\delta_{ij}  x^\mu x^\nu  \right).
		    \end{equation*} Moreover, for all $\xi \in \mathfrak{g}$, \begin{equation*}
		        \langle \tilde{F}_\infty, [\xi, F_\infty]\rangle (0) = 0
		    \end{equation*}
	\end{theorem} \begin{proof} Using \cref{estimategeneralbubbletree}, part of the proof of \cref{mainthm2} can be still be carried on in the general bubble-tree case. In this context, \cref{aux5} is still valid, that is
	    \begin{equation} 
   0 \equiv  \pi^2 F^{\infty} \circ \tilde{F}^{\infty}(0) - \int_{\mathrm{B}_{1/\delta}} \left(\frac{1}{2} \langle  S_{F_{\widehat{A}_k}} , \nabla\gamma\rangle \otimes r \diff r -S_{F_{\widehat{A}_k}} \circ \gamma + \frac{1}{4}\langle S_{F_{\widehat{A}_k}}, \gamma \rangle \xi \right)\mathrm{vol} + \widetilde{\varpi}_{k,\delta}\label{aux6}
\end{equation} where \begin{equation*}
    \lim_{\delta\to 0}\lim_{k\to +\infty} \widetilde{\varpi}_{k,\delta} = 0.
\end{equation*} In the general case, an additional step is required to deal with the integral term. The bubble-tree convergence implies that, in the (vector-valued) measure $S_{F_{\widehat{A}_k}} \mathrm{vol}$ on $\R^4$ weakly converges to \begin{equation*}
    S_{F_{\widehat{A}_\infty}} \mathrm{vol} + \sum_{\omega \in \mathcal{B}, A\ll \omega } \left(\int_{\R^4}S_{F_{\omega}} \mathrm{vol} \right)\delta_{p_{\omega}}
\end{equation*} where $p_{\omega}$ is the concentration point and the curvature of the bubble $\omega$ on $\widehat{A}_\infty$. However, for every pair of indices $(\alpha,\beta)$, \begin{align*}
    \int_{\R^4}S_{F_{\omega}}(\partial_\alpha, \partial_\beta) \mathrm{vol} =  \int_{\R^4}\langle S_{F_{\omega}}, \nabla^{\mathrm{s}}(x^\alpha \diff x^\beta)\rangle \mathrm{vol} =0
\end{align*} because $S_{F_{\omega}}$ is divergence-free. This proves that in fact $S_{F_{\widehat{A}_k}} \mathrm{vol}$ converges in the weak sense to $  S_{F_{\widehat{A}_\infty}} \mathrm{vol} $ and that the limit of \cref{aux6} when $k$ goes to $+\infty$ and $\delta$ goes to $0$ is the same as the one in the single bubble case, given in \cref{mainthm3}. The argument is even simpler for the gauge symmetry obstruction, since there is no corresponding integral term in the proof of \cref{mainthm5}.
	\end{proof}
	\begin{remark} \label{mainrmk} If $A_\infty \neq A^{0,1}$, that is if $A_\infty$ is itself a bubble, $\mathrm{W}=0$. If $\widehat{A}_\infty$ is SD or ASD, $S_{\widehat{F}^\infty}=0$ by \cref{stressSDASD}. In both cases, the first of the previous constraints become simply  \begin{equation*}
		         \langle F^{\infty}_{\alpha[i} , \tilde{F}^{\infty}_{j]\alpha} \rangle (0) + \delta_{ij} \langle F^{\infty}, \tilde{F}^{\infty}\rangle (0) = 0.
		    \end{equation*} Equivalently, we get \begin{align*}
		        \left(F^{\infty} \circ \tilde{F}^{\infty} -   \tilde{F}^{\infty} \circ F^{\infty} \right) (0) &= 0\\
		          \langle F^{\infty}, \tilde{F}^{\infty}\rangle (0)&= 0.
		    \end{align*}
	\end{remark}
	\section{Bubbling obstructions}
	\label{bbobstruction}
	
	 \begin{remark}
	     We use $\Lambda^+\R^4\otimes \mathfrak{g} \simeq \mathrm{Hom}(\mathfrak{g}, \Lambda^+\R^4)$ to identity $\mathfrak{g}$-valued a self-dual $2$-form $F$ with a map $\mathfrak{g}\to \Lambda^+\R^4$, explicitly given by $q \mapsto \langle q , F\rangle_{\mathfrak{g}}$.
	 \end{remark}
	 \begin{lemma}  \label{lemmasingv} Consider a sequence $(A_k)_k$ of Yang-Mills connections with bubble-tree limit $\mathcal{B}$. Fix $A_\infty, \widehat{A}_\infty\in \mathcal{B}$ such that $A_\infty< \widehat{A}_\infty$. Assume that $A_\infty$ is SD and $\widehat{A}_\infty$ is ASD. Compatibility conditions from \cref{mainthm4,mainthm5} are equivalent to the fact that $ \tilde{F}^\infty (F^\infty)^*:\Lambda^+\R^4\to \Lambda^+\R^4$ is traceless symmetric and $(F^\infty)^* \tilde{F}^\infty : \mathfrak{g} \to \mathfrak{g}$ is orthogonal to $\mathrm{ad}_\xi$ for all $\xi \in \mathfrak{g}$.
	 \end{lemma}
	 \begin{remark} When $\mathfrak{g}=\mathfrak{su}(2)$, $\{\mathrm{ad}_\xi, \xi \in \mathfrak{g}\}$ is precisely the set of skew-symmetric transformations of $\mathfrak{g}$ hence the last condition is saying that $(F^\infty)^* \tilde{F}^\infty$ is also symmetric.
	 \end{remark}
	
	A first consequence is the extension of the obstruction obtained in \cite{yin_blow-up_2023} for the blow-up of instantons bubbles of degree $\pm 1$.

\begin{theorem}
	\label{branchthm} Consider a sequence $(A_k)_k$ of Yang-Mills connections on $(M^4,h)$ with bubble-tree limit $\mathcal{B}$ as in \cref{bubbletreethm}. Fix $A_\infty, \widehat{A}_\infty\in \mathcal{B}$ such that $A_\infty< \widehat{A}_\infty$.  Assume that $A_\infty$ and $\widehat{A}_\infty$ are both $\pm 1$-instantons on $\Sph^4$. Then their either both SD or both ASD. More generally, if we have a chain of bubbles $A_\infty^{i,j} < \dots < A_\infty^{i,l}$ which are all $\pm1$-instantons, if $A_\infty^{i,j} $ is SD (resp. ASD) then they all are.
\end{theorem}
	\begin{proof} Suppose by contradiction that $A_\infty$ is SD and $\widehat{A}_\infty$ is ASD.  Recall the well-known expression for the curvature of the $1$-instanton on $\Sph^4$: it is pointwise of the form \[ F = \frac{|F|}{\sqrt{3}}\sum_{i=1}^3 q_i \theta^i\] where $|F|\neq 0$,  $\{q_i\}_i$ is an orthonomal basis of $\mathfrak{su}(2)$ and $\{\theta^i\}_i$ is an orthonomal basis on $\Lambda^+\R^4$. 	
		Therefore, with the notations of \cref{mainthm4}, as linear maps $ \mathfrak{su}(2) \to \Lambda^+\R^4$, $F^\infty(0)$ and $\tilde{F}^\infty(0)$ are conformal, hence $s=(F^\infty(0))^* \tilde{F}^\infty (0)$ is a conformal linear transformation of the $3$-dimensional space $\mathfrak{su}(2)$. By \cref{lemmasingv}, $s$ must be symmetric and traceless. However, all those constraints cannot be met at the same time: by conformality and symmetry, there exists $\alpha > 0$ and an orthogonal symmetry $\bar{s}$ such that $s = \alpha \bar{s}$. Since $\bar{s}$ is a symmetry, $\mathrm{Tr}(\bar{s}) = 3 - 2 \,\mathrm{dim \, ker}(\bar{s}+ \mathrm{id}) \neq 0$ which contradicts the fact that $s$ is also traceless. The general case, follows from a straightforward induction.
	\end{proof}

In particular, when $(M^4,h)$ is the round sphere $\Sph^4$ the limit solution can be considered as a bubble.  We deduce have the following corollary:
\begin{corollary} Consider a sequence $(A_k)_k$ of Yang-Mills connections on the round $\Sph^4$ with bubble-tree limit $\mathcal{B}$ as in \cref{bubbletreethm}. Assume that $\mathcal{B}$ consists only of $\pm1$-instantons then if $A_\infty$ is SD (resp. ASD) then all the bubbles are SD (resp. ASD).
\end{corollary}

A distinctive feature of \cref{mainthm} compared to \cite[Theorem 1.4]{yin_blow-up_2023} is that no local conformal flatness assumption is required at the concentration point. We illustrate this on $\mathbf{CP}^2$ endowed with its Fubini-Study metric, whose Weyl tensor does not vanish. Using the explicit description of $\mathrm{SU}(2)$-instantons on $\mathbf{CP}^2$ due to Groisser \cite{groisser1990geometry} (see also Habermann \cite{habermann1992family}), we rule out the following bubbling configuration.
	\begin{proposition} Consider a sequence $(A_k)_k$ of Yang-Mills connections the round $\Sph^4$ with bubble-tree limit $\mathcal{B}$ as in \cref{bubbletreethm}. Assume that $A^{0,1}$ is a $1$-instanton. Then no bubble $A^{i,j}$ such that $A^{0,1}<A^{i,j}$ can be a $-1$-instanton.
	\end{proposition}
	\begin{proof} 
		Let's proceed again by contradiction. 	Up to gauge and up to the action of the isometry group $\mathrm{SU}(3)$, the curvature $F_{A^{0,1}}$ is of the form  \begin{equation*}
		F_{A^{0,1}} = 2\frac{1-t^2}{(D-t^2)^2}\left( -D^2\omega\, \mathbf{i} + t\, \diff z^1 \wedge \diff z^2 \,\mathbf{j}\right)
		\end{equation*} where $t\in[0,1)$, $\omega$ is the Kähler form and $D= 1 + |z^1|^2 + |z^2|^2$. In particular, since $|\omega|=\sqrt{2}$, $\beta = t |\diff z^1 \wedge \diff z^2|/(2D^2)= t(2\sqrt D)^{-1}$. Write $f$ and $\tilde{f}$ for the maps $\mathfrak{su}(2)\to \Lambda^+\R^4$ associated to $F_{A^{0,1}}$ and $\psi^* F_{A^{i,j}}$ and define $s=f ^* \tilde{f}$. By \cref{lemmasingv}, $s$ is traceless and symmetric. Since $\tilde{f}$ is conformal, $s^2= s s^* =\alpha^2 f ^* f$ for some $\alpha >0$. From the previous expression, the eigenvalues of $f ^* f $ are, up to a common multiplicative factor, $1,\beta^2,\beta^2$. We deduce that the (real) eigenvalues of $s$ must be $1,\beta,\beta$ up to signs. The sum of these eigenvalues must vanish but this is impossible since $\beta <1/2$.
	\end{proof}

	\appendix
	\section{Appendix}
	
	\begin{lemma} \label{formsinteriorproduitscalaire} If $(V^n,\langle \cdot, \cdot\rangle)$ is a euclidean space, for every $a, b \in \Lambda^k V^*$ and  $X,Y\in V$, \begin{equation*} \langle \iota_X a , \iota_Y b\rangle + \langle \iota_Y \star a , \iota_X \star b\rangle = \langle X, Y \rangle \langle a,b\rangle \end{equation*}
	\end{lemma}

\begin{proof}
     		Identify the vectors $X,Y$ with the $1$-forms $\langle X, \cdot \rangle$ and $\langle Y, \cdot \rangle$ respectively. Recall that $\iota_X (\cdot)$ and $X\wedge (\cdot)$ are adjoints of each other. Notice that this implies the formula $\star\iota_X= X\wedge \star$. Therefore, since $\star$ is an isometry, we deduce the following chain of equalities \begin{align*} 
     		\langle \iota_X a \wedge  \iota_Y  b\rangle &=  \langle \star \iota_X a, \star \iota_Y b \rangle\\ 			
     		&= \langle X \wedge \star a, Y \wedge \star b\rangle\\ 			
     		& =   \langle \star a, \iota_X(Y \wedge \star b)\rangle\\  
     		&= \langle \star a, \langle X, Y \rangle \star b - Y\wedge \iota_X \star b)\rangle \\
     	& = \langle X, Y \rangle \langle a,b\rangle - \langle \iota_Y \star a , \iota_X \star b\rangle
     		\end{align*} and the proof is complete.
\end{proof}

\subsection{A Poisson equation on degenerating annuli}

Let $\alpha$ a positive real number that is not an integer. Denote by
$\mathcal{H}_{\alpha}$ the space of harmonic polynomials in $\mathbf{R}^4$ of
degree less than $\alpha - 1$ and
\begin{equation*} \overline{\mathcal{H}}_{\alpha} = \{ x \mapsto P (x) + | x |^{- 2} Q (x / |
x |^2), P, Q \in \mathcal{H}_{\alpha} \} . \end{equation*}
We define for $\lambda \in (0, 1)$, $x \in \mathrm{B}_1 \backslash
\mathrm{B}_{\lambda}$, $\omega_{\lambda} (x) = | x | + \lambda | x |^{- 1}$.

\begin{lemma}
	\label{harmonicestimate}$\overline{\mathcal{H}}_{\alpha}$ is the space of
	all harmonic functions $\vartheta$ in $\mathbf{R}^4 \backslash \{ 0 \}$ satisfying
	\begin{equation*} \sup_{x \neq 0} \omega_1 (x)^{- \alpha} | x \vartheta (x) | < + \infty \end{equation*}
	Moreover the map $\Phi: \overline{\mathcal{H}}_{\alpha} \rightarrow
	\mathcal{L} (\mathcal{H}_{\alpha}, \mathbf{R}^2)$ defined by
	\begin{equation*} \Phi (\vartheta) \varphi = \int_{\mathbf{S}^3} (\vartheta(x), \partial_{r} \vartheta(x)) \varphi(x)\diff x \end{equation*}
	is an isomorphism.
\end{lemma}

\begin{proof}
	The first part is a Liouville type result (see for instance \cite[Theorem 10.5]{axler2013harmonic}. For the second part, it is enough to check that $\ker \Phi = \{ 0 \}$ since
	$\Phi$ is a linear map between spaces of the same dimension. If $\vartheta \in \ker
	\Phi$, taking $\varphi = \vartheta_{|\mathbf{S}^3}$ and $\varphi = x^k \partial_k
	\vartheta_{|\mathbf{S}^3}$ we get
	\begin{equation*} \int_{\mathbf{S}^3 }  | \vartheta |^2 = \int_{\mathbf{S}^3} | \partial_{r} \vartheta
	|^2 = 0 \end{equation*}
	However this would imply that $w = \vartheta\mathbf{1}_{\text{B}_1}$ is harmonic
	in $\mathbf{R}^4 \backslash \{ 0 \}$ then by analyticity that $w = 0$ and
	finally that $\vartheta = 0$ which is the desired contradiction.
\end{proof}

\begin{proposition} \label{ellipticpdev0}
	There exists $\lambda_0 \in (0, 1)$ and $C > 0$ such that for all $\lambda \in
	(0, \lambda_0)$, for all $v \in \mathrm{W}^{1, 2}_{\mathrm{loc}} (\mathrm{B}_1
	\backslash \mathrm{B}_{\lambda})$ satisfying the following conditions
	\begin{enumerate}
		\item $| x |^3 | \Delta v | \leq \omega_{\lambda}^{\alpha}$ in
		$\mathrm{B}_1 \backslash \mathrm{B}_{\lambda}$ \label[hyp]{ci}
		
		\item $| x | | v | \leq 1$ in $\mathrm{B}_1 \backslash
		\mathrm{B}_{\lambda}$ \label[hyp]{cii}
		
		\item $\Phi \left( v_{| \partial \mathrm{B}_{\sqrt{\lambda}}} \right) =
		0$ \label[hyp]{ciii},
	\end{enumerate}
	we have for all $x \in \mathrm{B}_1 \backslash \mathrm{B}_{\lambda}$
	\begin{equation*} | x | | v | \leq C \omega_{\lambda}^{\alpha} . \end{equation*}
\end{proposition}

\begin{proof}
	We use a contradiction argument. Assume there exists a sequence
	$(\lambda_k)_k \subset (0, 1)$ and $\lim \lambda_k = 0$ and functions
	$(v_k)_k$ with $v_k \in \mathrm{W}^{1, 2}_{\mathrm{loc}} (\mathrm{B}_1
	\backslash \mathrm{B}_{\lambda_k})$ satisfying \cref{ci,cii,ciii} and
	\begin{equation*} M_k \coloneq \sup_{\text{B}_1 \backslash \text{B}_{\lambda_k}}
	\omega_{\lambda_k}^{- \alpha} | x | | v_k | \underset{k \rightarrow +
		\infty}{\rightarrow} + \infty \end{equation*}
	Consider $x_k \in \text{B}_1 \backslash \text{B}_{\lambda_k}$ realising the
	supremum, that is such that $r_k | v_k (x_k) | = M_k \omega_{\lambda_k}
	(x_k)^{\alpha}$, where $r_k = | x_k |$. By \cref{ci}, we have for all
	$\delta > 0$
	\begin{equation*} \sup_{\text{B}_1 \backslash \text{B}_{\delta} \cup \text{B}_{\lambda_k /
			\delta} \backslash \text{B}_{\lambda_k}} \omega_{\lambda_k}^{- \alpha} | x | | v_k |
	\leq \left(\inf_{\rho \in (\delta, 1) \cup (\lambda_k, \lambda_k / \delta)} (\rho
	+ \lambda_k \rho^{- 1})\right)^{- \alpha} \leq \delta^{- \alpha} \end{equation*}
	which implies that both inequalities $\limsup r_k > \delta$ and $\liminf r_k /
	\lambda_k < 1 / \delta$ are false. This in turn implies $r_k \rightarrow 0$
	and $\lambda_k / r_k \rightarrow 0$.
	
	\noindent Define now for $y \in \text{B}_{1 / r_k} \backslash \text{B}_{\lambda_k /
		r_k}$ the rescaled function $\overline{v}_k (y) = | v_k (x_k) |^{- 1} v_k
	(r_k y)$. The definition of $M_k$ translates as
	\begin{align}
		| y | | \overline{v}_k (y) |  \leq & \frac{r_k | y |}{M_k} \frac{| v_k
			(r_k y) |}{(r_k + \lambda_k r_k^{- 1})^{\alpha}}\notag\\
		 \leq & \frac{(r_k | y | + \lambda_k (r_k | y |)^{- 1})^{\alpha}}{(r_k +
			\lambda_k r_k^{- 1})^{\alpha}}\notag\\
		\leq & \left( | y | \frac{r_k}{r_k + \lambda_k r_k^{- 1}} + | y |^{- 1}
		\frac{\lambda_k r_k^{- 1}}{r_k + \lambda_k r_k^{- 1}} \right)^{\alpha}\label{vkest}
	\end{align}
	whereas assumption $\left( \text{\textit{i}} \right)$ gives \begin{align*}
	| y |^3 | \Delta \overline{v}_k | & \leq  | y |^3 r_k^2 | v_k (x_k)
|^{- 1} | (\Delta v_k) (r_k y) |\\
& \leq  \frac{| r_k y |^3}{M_k (r_k + \lambda_k r_k^{- 1})^{\alpha}}
| (\Delta v_k) (r_k y) |\\
& \leq  \frac{1}{M_k} (| y | + | y |^{- 1})^{\alpha}
	\end{align*}
	This means
	\begin{equation} \Delta \overline{v}_k  \underset{k \rightarrow + \infty}{\rightarrow} 0
	\text{ in } \text{L}^{\infty}_{\mathrm{loc}} (\mathbf{R}^4 \backslash \{ 0
	\}) \end{equation}
	and we deduce that there exists a harmonic function $\overline{v}_{\infty}$
	on $\mathbf{R}^4 \backslash \{ 0 \}$ such that
	\begin{equation} \overline{v}_k  \underset{k \rightarrow + \infty}{\rightarrow}
	\overline{v}_{\infty} \text{ in } \mathcal{C}^1_{\mathrm{loc}}
	(\mathbf{R}^4 \backslash \{ 0 \}). \end{equation} Moreover,
	since $\overline{v}_k (x_k / r_k) = 1$, we get in the limit
	\begin{equation} \inf_{\partial \text{B}_1}  | \overline{v}_{\infty} | \geq 1. \label{nonvanishingvlimit}\end{equation}
	Now we are in either one of the following three situations:\begin{itemize}
	    \item $\limsup {r_k}/{\sqrt{\lambda_k}} =
	+ \infty$. In this case, there is a subsequence satisfying $r_k^2 /\lambda_k
	\rightarrow + \infty$. According to \cref{vkest} for $y \in \text{B}_{1
		/ r_k} \backslash \text{B}_{\lambda_k / r_k}$
	\begin{equation} | y | | \overline{v}_k (y) | \leq \left( | y | \frac{r_k^2 \lambda_k^{-
			1}}{r_k^2 \lambda_k^{- 1} + 1} + | y |^{- 1} \frac{1}{r_k^2 \lambda_k^{-
			1} + 1} \right)^{\alpha} \end{equation} so we deduce that for all $y \in \mathbf{R}^4 \backslash \{ 0 \}, | y | |
	\overline{v}_{\infty} (y) | \leq | y |^{\alpha}$ that is
	\begin{equation} | \overline{v}_{\infty} (y) | \leq | y |^{\alpha - 1}. \end{equation}
	Since $\alpha$ is not an integer, the only possibility is $\overline{v}_\infty = 0$ which contradicts \cref{nonvanishingvlimit}.
	\item $\liminf {r_k}/{\sqrt{\lambda_k}} =
	0$. A similar argument yields \begin{equation} | \overline{v}_{\infty} (y) | \leq | y |^{-\alpha - 1} \end{equation}
for all $y \neq 0$, which is again a contradiction.
\item $r_k/\sqrt{\lambda_k}$ is bounded and bounded away from zero. This implies that, up to a subsequence, the ratio $r_k / \sqrt{\lambda_k}$
	converges to some $\rho \in (0, + \infty)$. From \cref{vkest}, we deduce 	\begin{equation} | y | | \overline{v}_{\infty} (y) | \leq (| y | + | y |^{- 1})^{\alpha}
	\end{equation} that is $ \overline{v}_{\infty} \in \overline{\mathcal{H}}_\alpha$ and assumption $\left(
	\text{\textit{iii}} \right)$ gives in the limit $\Phi \left(
	{\overline{v}_{\infty}}_{| \partial \mathrm{B}_{\rho}} \right) = 0$ so
	$\overline{v}_{\infty} = 0$ since $\Phi$ is one to one.
	\end{itemize}
\end{proof}

\begin{theorem} \label{ellipticPDEannulus}
	There exists $\lambda_0 \in (0, 1)$ and $C > 0$ such that for all $\lambda \in
	(0, \lambda_0)$, for all $v \in \mathrm{W}^{1, 2}_{\mathrm{loc}} (\mathrm{B}_1
	\backslash \mathrm{B}_{\lambda})$ satisfying the following conditions
	\begin{enumerate}
		\item $| x |^3 | \Delta v | \leq \omega_{\lambda}^{\alpha}$ in
		$\mathrm{B}_1 \backslash \mathrm{B}_{\lambda}$ \label{cci}
		
		\item $| x | | v | \leq \omega_{\lambda}^{\lfloor
			\alpha \rfloor}$ in $\mathrm{B}_1 \backslash \mathrm{B}_{\lambda}$
		\label{ccii}
	\end{enumerate}
	we have for all $x \in \mathrm{B}_1 \backslash \mathrm{B}_{\lambda}$
	\begin{equation*} | x | | v - \vartheta | \leq C \omega^{\alpha} \end{equation*}
	where $\vartheta$ is the unique element of $\overline{\mathcal{H}}_{\alpha}$ such
	that $\Phi \left( \vartheta_{| \partial \mathrm{B}_{\sqrt{\lambda}}} \right) = \Phi
	\left( v_{| \partial \mathrm{B}_{\sqrt{\lambda}}} \right)$.
\end{theorem}

\begin{proof}
	We apply \cref{ellipticpdev0} to $v - \vartheta$ which satisfies \cref{ci,ciii} trivially. 

\noindent Consider the norm $\| \cdot \|$ on $\mathcal{H}_{\alpha}$ defined by $\|
	\varphi \| = \sup_{\mathbf{S}^3} (| \varphi | + | \partial_{r} \varphi |)$
	and denote by $\Vert \cdot \Vert$ the associated operator norm
	for linear maps $\mathcal{H}_{\alpha} \rightarrow \mathbf{R}^2$. Since
	$\Phi$ is an isomorphism, the map $\vartheta \mapsto \Vert \Phi (\vartheta)
	\Vert$ defines a norm on $\overline{\mathcal{H}_{\alpha}}$ which is
	equivalent to $\vartheta \mapsto \sup_{x \neq 0} \omega_1 (x)^{- \lfloor \alpha
		\rfloor} | x \vartheta (x) |$ since $\overline{\mathcal{H}_{\alpha}}$ is finite
	dimensional. In particular for all $x \neq 0$
	\begin{equation*} \omega_1 (x)^{- \lfloor \alpha \rfloor} \left| \sqrt{\lambda} x \vartheta \left(
	\sqrt{\lambda} x \right) \right| \leq C \Vert \Phi \left(
	\sqrt{\lambda} \vartheta_{| \partial \mathrm{B}_{\sqrt{\lambda}}} \right)
	\Vert = C \Phi \left( \sqrt{\lambda} v_{| \partial
		\mathrm{B}_{\sqrt{\lambda}}} \right) \leq C \lambda^{\lfloor \alpha
		\rfloor / 2} \end{equation*}
	that is
	\begin{equation*} | x \vartheta (x) | \leq C \lambda^{\lfloor \alpha \rfloor / 2} \left( \frac{| x
		|}{\sqrt{\lambda}} + \frac{\sqrt{\lambda}}{| x |} \right)^{\lfloor \alpha
		\rfloor} = C \omega_{\lambda}^{\lfloor \alpha \rfloor} \end{equation*}
	and in particular, for $x \in \mathrm{B}_1 \backslash \mathrm{B}_{\lambda}$,
	\begin{equation*} | x | | v (x) - \vartheta (x) | \leq C \end{equation*}
	which means that \cref{cii} is satisfied (up to a multiplicative
	constant).
\end{proof}

\begin{remark} \label{rkellipticPDEannulus}
	Elliptic regularity estimates yield, under the hypothesis of \cref{ellipticPDEannulus}, for $x\in \mathrm{B}_{1/2}\backslash \mathrm{B}_{2\lambda}$,
	\begin{equation*} | x | | v -\vartheta | + | x |^2 | \nabla v -\nabla \vartheta| \leq C \omega^{\alpha} . \end{equation*}
\end{remark}

\begin{lemma} \label{flatornot}
	There exists $C > 0$ such that, if $h$ is a metric on the unit ball with
	\begin{equation*} | h - \xi | + | x | | \nabla h | \leq | x |^2, \end{equation*}
	where $\xi$ is the flat metric, then for any $1$-form $A$ such that, in
	$\mathrm{B}_1 \backslash \mathrm{B}_{\lambda}$, $\lambda \in (0, 1)$,
	\begin{equation*} | x |^2 | \nabla A | + | x |^3 | \nabla^2 A | \leq \omega^2 \end{equation*}
	we have
	\begin{equation*} | x |^3 | \Delta_{\xi} A - \Delta_h A | \leq C \omega^4 . \end{equation*}
\end{lemma}

\begin{proof}
	For any 2-form $\Omega$,
	\begin{eqnarray*}
		\diff^{\ast_h} \Omega & = & - \frac{1}{2} \star_h \diff \star_h
		(\Omega_{i j} \diff x^i \wedge \diff x^j)\\
		& = & - \frac{1}{2} \star_h (\partial_k \Omega_{i j} \diff x^k \wedge
		\star_h (\diff x^i \wedge \diff x^j) + \Omega_{i j} \diff \star_h
		(\diff x^i \wedge \diff x^j))\\
		& = & - \frac{1}{2} \partial_k \Omega_{i j} \star_h (\diff x^k \wedge
		\star_h (\diff x^i \wedge \diff x^j)) + \frac{1}{2} \Omega_{i j}
		\diff^{\ast_h} (\diff x^i \wedge \diff x^j)
	\end{eqnarray*}
	It is clear that
	\begin{eqnarray*}
		| \partial_k \Omega_{i j} \star_h (\diff x^k \wedge \star_h (\diff x^i
		\wedge \diff x^j)) - \partial_k \Omega_{i j} \star_{\xi} (\diff x^k
		\wedge \star_{\xi} (\diff x^i \wedge \diff x^j)) | & \leq & C | \nabla
		\Omega | | h - \xi |\\
		| \Omega_{i j} \diff^{\ast_h} (\diff x^i \wedge \diff x^j) - \Omega_{i
			j} \diff^{\ast_{\xi}} (\diff x^i \wedge \diff x^j) | & \leq & C |
		\Omega | (| h - \xi | + | \nabla h |)
	\end{eqnarray*}
	therefore for $A$ as in the statement of the lemma,
	\begin{eqnarray*}
		| x |^3 | \diff^{\ast_h} \diff A - \diff^{\ast_{\xi}} \diff A | & \leq
		& C (| \nabla^2 A | | h - \xi | + | \nabla A | (| h - \xi | + | \nabla h
		|))\\
		& \leq & C \omega^2 | x |^2\\
		& \leq & C \omega^4 .
	\end{eqnarray*}
	
\end{proof}

	\addcontentsline{toc}{section}{References}
\printbibliography

\end{document}